\documentclass[a4paper]{article}

\usepackage[margin=1in]{geometry}

\usepackage{latexsym}

\usepackage{amsmath, amsthm}
\usepackage{dsfont}
\usepackage{upgreek}
\usepackage{tikz}
\usepackage{times}
\usepackage{amssymb}
\usepackage[font={small}]{caption}
\usepackage{mathrsfs}
\usepackage{textcomp} 
\usepackage[font={small},caption=false]{subfig}
\usepackage{mathtools}
\usepackage{booktabs}
\usepackage{cancel}
\usepackage{graphicx}
\usetikzlibrary{automata,positioning}
\usepackage{pgfplots}

\makeatletter
\newtheorem*{rep@theorem}{\rep@title}
\newcommand{\newreptheorem}[2]{%
\newenvironment{rep#1}[1]{%
 \def\rep@title{#2 \ref{##1}}%
 \begin{rep@theorem}}%
 {\end{rep@theorem}}}
 
\def\@seccntformat#1{\@ifundefined{#1@cntformat}%
   {\csname the#1\endcsname\quad}  
   {\csname #1@cntformat\endcsname}
}
 
\let\oldappendix\appendix 
\renewcommand\appendix{%
    \oldappendix
    \newcommand{\section@cntformat}{\appendixname~\thesection.\quad}}
\makeatother

\usepackage{savesym}
\savesymbol{iint}
\usepackage{txfonts} 
\restoresymbol{TXF}{iint}
\usepackage{ucs}
\usepackage[ansinew]{inputenc}
\usepackage{textcomp}

\newtheorem{theorem}{Theorem}
\newtheorem{proposition}{Proposition}
\newtheorem{corollary}{Corollary}
\newtheorem{lemma}{Lemma}
\newtheorem*{remark}{Remark}
\theoremstyle{definition}
\newtheorem{example}{Example}[section]
\newtheorem{definition}{Definition}

\newcommand*\Heq{\ensuremath{\overset{\kern2pt H}{=}}}

\newcommand{\defeq}{\stackrel{\text{\tiny def}}{=}}

\usepackage{url}
\urldef{\mails}\path|{m.talebi, j.f.groote, j.p.linnartz}@tue.nl|

\begin{document}

\title{The Mean Drift: Tailoring the Mean Field Theory of Markov Processes for Real-World Applications}

\author{Mahmoud Talebi\thanks{The authors would like to thank Erik de Vink, Mieke Massink, Tjalling Tjalkens and Ulyana Tikhonova for their constructive comments and helpful remarks.}~~~~~Jan Friso Groote ~~~~~Jean-Paul M.G. Linnartz\\
\mails\\\\
\small Eindhoven University of Technology\\
\small Den Dolech 2, 5612 AZ, Eindhoven, The Netherlands\\}
\date{}

\maketitle

\begin{abstract}
\noindent  The statement of the mean field approximation theorem in the mean field theory of Markov processes particularly targets the behaviour of population processes with an unbounded number of agents. However, in most real-world engineering applications one faces the problem of analysing middle-sized systems in which the number of agents is bounded. In this paper we build on previous work in this area and introduce the mean drift. We present the concept of population processes and the conditions under which the approximation theorems apply, and then show how the mean drift is derived through a systematic application of the propagation of chaos. We then use the mean drift to construct a new set of ordinary differential equations which address the analysis of population processes with an arbitrary size.

\end{abstract}

\section{Introduction}
Population processes are stochastic models of systems which consist of a number of similar agents (or particles)\cite{kurtz1981approximation}. When the impact of each agent on the behaviour of the system is similar to other agents, it is said that the population process is a mean field interaction model \cite{benaim2008class}. It is possible to apply a symmetric reduction on the state space of these types of processes and gain some efficiency in their analysis. Mean field approximation refers to the continuous, deterministic approximations of the behaviour of such processes in their first moment, when the number of agents grows very large. These approximations were first proposed for several concrete cases in various areas of study e.g., from as early as the 18th century in population biology, where models such as the predator-prey equations and the SIR equations are used to describe the balance of species in an ecosystem and the dynamics of epidemics respectively~\cite{berryman1992orgins,dietz2002daniel}.

Since then, general theorems have been proven which show the convergence of the behaviour of population processes to solutions of differential equations. The proofs follow roughly the same steps which generally rely on Gr\"onwall's lemma and martingale inequalities \cite{darling2008differential}. One of the first generalized approximation theorems was given by Kurtz \cite{kurtz1970solutions}. The theory gives conditions which define a family of such models called density-dependent population processes, and finds their deterministic approximations by a theorem which is generally called the law of large numbers for standard Poisson processes~\cite{ethier2009markov}.

The mean field theory of Markov processes is increasingly being applied in the fields of computer science and communication engineering. In the field of communication engineering and starting with Bianchi's analysis of the IEEE 802.11 DCF protocol~\cite{bianchi1998ieee,bianchi2000performance}, much research has focused on discussing the validity of the so-called decoupling assumption in this analysis~\cite{bordenave2005random,
vvedenskaya2007multiuser,
sharma2009performance,cho2012asymptotic}. Several general frameworks have also been proposed which target the analysis of computer and communication systems~\cite{benaim2008class,le2007generic}. In the field of computer science the initial application of the approximations was intuitively motivated by methods such as fluid and diffusion approximations of queueing networks~\cite{hillston2005fluid}. These have resulted in the development of methods and tools to automate the analysis of mean field models, with extensive progress in the context of the stochastic process algebra PEPA \cite{hayden2010fluid,
hayden2011scalable,
bortolussi2013continuous} among others \cite{bortolussi2012fluid}. However, still a large family of models are deemed unsuitable for this fluid approximation analysis, since they often lead to demonstrably inaccurate approximations~\cite{hayden2010fluid,pourranjbar2012don}. This calls for revisiting the fundamental roots of the theory. Such an approach has been taken in~\cite{beccuti2014analysis} where the authors use a set of extended diffusion approximations to derive precise approximations of stochastic Petri Net models.

We identify the current challenge as the problem of analysing middle-sized systems: systems which are so large that they suffer from {\it state space explosion}, but not large enough such that they can be accurately analysed by common approximation methods. In this paper we focus on the evaluation of these middle-sized systems, and the most important contribution is the introduction of Poisson mean of intensities (equation~(\ref{eq:poaverage})), and the way they relate to the approximation theorem based on the idea of propagation of chaos.

Through equation (\ref{eq:poaverage}) we express the idea that occupancy measures can be seen as Poisson arrival rates. Based on this observation we propose employing the concept of the {\it Poisson mean} of the drift to build the set of ordinary differential equations in (\ref{eq:betterode}) when dealing with bounded systems. To provide further proof for the consistency of our observation with respect to already established results in the mean field theory, we show that the drift and the mean drift are equivalent in the limit (Theorem~\ref{theo:drifts}).

In the current text, we use only the most essential and trivial concepts to build the mean field approximations. Several decisions were made to achieve this goal: the choice of discrete-time Markov processes to specify the agent behaviour, a detailed discussion on time, the discussion on the derivation of transition maps, and the presentation of the convergence result in two clear steps (following~\cite{benaim2008class}). In our opinion, this attitude is essential in promoting a correct understanding of the theory.

The rest of this text is organized as follows. In Section~\ref{sec:pre} we give a brief introduction to necessary concepts in real analysis and probability theory which are used in later sections. Section~\ref{sec:method} describes the family of mean field interaction models, and the derivation of their deterministic approximations. We then state Theorem~\ref{theo1} and Corollary~\ref{cor1}, which provide the basis for stating the main results of this work. Finally, in Section~\ref{sec:propagation} we present the idea of mean drift as our main result. The detailed proofs follow in Section~\ref{sec:proofs}.

\section{Preliminaries}\label{sec:pre}
In this section we give essential definitions and theorems which will be needed in later sections. In the definitions related to probability theory we rely mostly on \cite{kallenberg2006foundations}. The definitions of independence, conditional expectations and the law of large numbers are from \cite{billingsley2008probability}. Several properties of Poisson point processes and martingales are according to \cite{kurtz2007lectures}. For the definitions and theorems related to real analysis (including the Lebesgue integral) we refer mainly to \cite{royden2010real}, except that the statement of the Gr\"onwall's inequality is based on \cite{kurtz2007lectures}.

In what follows, we will use the following notations in the text. Let $T$ be a totally ordered set. We use the notation $\{X_t:t\in T\}$ to show a sequence of objects indexed by set $T$. The notation $\mathbb{R}_{\geq 0}=[0,\infty)$ is used to refer to the set of non-negative real numbers. In the same manner, $\mathbb{Q}_{\geq 0}$ is the set of non-negative rational numbers.

To make a clear distinction between superscripts and exponentiation, we write $X^N$ to denote ``$X$ raised to the power $N$'' and $X^{(N)}$ to denote ``$X$ annotated by $N$''.

\subsection{Measures and probability}\label{subsec:measure}
Let $X$ be a set, and $d:X\times X\rightarrow\mathbb{R}$ be a function which for all $x,y\in X$ satisfies the following properties:
\begin{itemize}
\item $d(x,y)\geq 0$, where $d(x,y)=0$ if and only if $x=y$,
\item $d(x,y)=d(y,x)$ (symmetry),
\item For all $z\in X$, $d(x,y)\leq d(x,z)+d(z,y)$ (triangle inequality),
\end{itemize}
Then $d$ is called a {\it distance function} or a {\it metric}, and the pair $(X,d)$ is a {\it metric space}.

A $\sigma$-algebra $\Sigma_{X}$ on a set $X$ is a set of subsets of $X$ which contains the empty set $\emptyset$ and is closed under the complement, countable union and countable intersection of subsets. The pair $(X,\Sigma_{X})$ is called a {\it measurable space}.

Let $X$ and $Y$ be sets. Consider the function $f:X\rightarrow Y$, $f^{-1}:2^{Y}\rightarrow 2^{X}$ is a set mapping for which for $B\in 2^{Y}$:
\[f^{-1}(B)\defeq\{x\in X\mid f(x)\in B\}.\]
$f^{-1}(B)$ is often called the source of $B$.
\begin{definition}[Measurable Function]\label{def:measurable}
Let $(X,\Sigma_X)$ and $(Y,\Sigma_Y)$ be measurable spaces, and let $f:X\rightarrow Y$ be a function. For all $B\in\Sigma_{Y}$, and the set mapping $f^{-1}$, if 
\[f^{-1}(B)\in\Sigma_{X},\]
$f$ is called a $\Sigma_X/\Sigma_Y$-measurable function.
\end{definition}
When the exact shape of the $\sigma$-algebras are not of concern or already clear in a context, one may also refer to these functions as {\it $\Sigma_X$-measurable} or just {\it measurable}.

Let $(X,\Sigma_X)$ and $(Y,\Sigma_Y)$ be measurable spaces. Let $f:X\rightarrow Y$ be a function. Then the {\it induced $\sigma$-algebra of $f$}, denoted by $\sigma(f)$ is the smallest $\sigma$-algebra $\mathcal{A}\subseteq \Sigma_X$ which makes $f$ measurable.

Let $(X,\Sigma_X)$ be a measurable space. A {\it measure} on $(X,\Sigma_X)$ is defined as a function $\mu:\Sigma_X\rightarrow \mathbb{R}_{\geq 0}$ which satisfies the following two properties:
\begin{itemize}
\item $\mu(\emptyset)=0$,
\item For all countable collections of pairwise disjoint sets $\left\{E_{i}\right\}_{i=1}^{\infty}$ in $\Sigma_{X}$:
\[\mu\left(\bigcup_{k=1}^{\infty}E_k\right)=\sum_{k=1}^{\infty}\mu(E_k).\]
\end{itemize}

Based on the definition of measures, we define {\it Lebesgue integrals}, which are indispensable tools in modern probability theory. Let $(X,\Sigma_X)$ be a measurable space and let $\mu:\Sigma_X\rightarrow\mathbb{R}_{\geq 0}$ be a measure. Let $A\in X$, for $x\in X$ the {\it indicator function} $\mathds{1}_A:X\rightarrow \mathbb{R}_{\geq 0}$ is defined as:
\[\mathds{1}_A(x)\defeq\left\{
  \begin{array}{l l}
    1 & \quad \text{if $x\in A$}\\
    0 & \quad \text{if $x\notin A$}.
  \end{array} \right.\]
For the boolean domain, we define $\mathds{1}:\mathbb{B}\rightarrow\mathbb{R}_{\geq 0}$, which for $b\in\mathbb{B}$:
\[\mathds{1}(b)\defeq\left\{
  \begin{array}{l l}
    1 & \quad \text{if $b$ true}\\
    0 & \quad \text{if $b$ false}.
  \end{array} \right.\]

Given a measure $\mu$ on the measurable space $(X,\Sigma_X)$, the Lebesgue integral of $\mathds{1}_A$ is denoted by $\int_{X}\mathds{1}_A\,d\mu$ which is:
\[\int_{X}\mathds{1}_A\,d\mu\defeq\mu(A).\]

More generally, for natural number $n$ let $A_1,\ldots,A_m\in \Sigma_X$ be sets, and let $c_1,\ldots,c_m\in\mathbb{R}_{\geq 0}^n$ be real vectors. A function $\textbf{S}:X\rightarrow\mathbb{R}_{\geq 0}^n$ where $\textbf{S}=\sum_k c_k\mathds{1}_{A_k}$ is called a {\it simple function}. The Lebesgue integral of $\textbf{S}$ is:
\[\int_{X}\textbf{S}\,d\mu\defeq\sum_k c_k\mu(A_k).\]

For an arbitrary function $f:X\rightarrow\mathbb{R}_{\geq 0}^n$ define:
\[\underline{\int_X} f\,d\mu\defeq\sup\left\{\int_X \textbf{S}\,d\mu: 0\leq \textbf{S}\leq f\text{~pointwise,~}\textbf{S}\text{~simple}\right\}\]
and
\[\overline{\int_X} f\,d\mu\defeq\inf\left\{\int_X \textbf{S}\,d\mu: \textbf{S}\geq f \text{~pointwise,~}\textbf{S}\text{~simple}\right\}.\]
The Lebesgue integral $\int_X f\,d\mu$ exists ($f$ is Lebesgue integrable) if $\overline{\int_X} f\,d\mu=\underline{\int_X} f\,d\mu<\infty$, in which case:
\[\int_X f\,d\mu\defeq\overline{\int_X} f\,d\mu.\]
We now define the Lebesgue integral of signed functions. Let $f:X\rightarrow\mathbb{R}^n$. There exist functions $f_1,\ldots,f_n$ such that for $x\in X$:
\[f(x)=(f_1(x),\ldots,f_n(x)).\]
For $1\leq i\leq n$ define the sequence of functions $f^{+}_i$ and $f^{-}_i$ which satisfy:
\begin{equation*}
f^{+}_i(x)=\begin{cases}
f_i(x)~~\text{if}~f_i(x)>0,\\
0~~~~~~~~\text{otherwise}
\end{cases}
\end{equation*}
and:
\begin{equation*}
f^{-}_i(x)=\begin{cases}
-f_i(x)~~~\text{if}~f_i(x)< 0,\\
0~~~~~~~~~~~~\text{otherwise}.
\end{cases}
\end{equation*}
Accordingly, define the functions $f^{+}:X\rightarrow\mathbb{R}_{\geq 0}^{n}$ and $f^{-}:X\rightarrow\mathbb{R}_{\geq 0}^{n}$ as:
\[f^{+}(x)\defeq(f^{+}_1(x),\ldots,f^{+}_n(x))~~\text{and}~~f^{-}(x)\defeq(f^{-}_1(x),\ldots,f^{-}_n(x)).\]
The Lebesgue integral of a signed function $f$ is:
\[\int_{X}f\,d\mu\defeq\int_{X}f^{+}\,d\mu-\int_{X}f^{-}\,d\mu,\]
which exists only if each of the integrals on the right hand side of this equation exist.

We have now introduced enough basic concepts to start talking about probabilities. For the measurable space $(X,\Sigma_{X})$, a measure $P:\Sigma_{X}\rightarrow \mathbb{R}_{\geq 0}$ is called a {\it probability measure} if it satisfies the property: $P(X)=1$.

\begin{definition}[Probability Space]
A triple $(\Omega, \mathcal{F}, P)$ is called a {\it probability space} in which:
\begin{itemize}
\item $\Omega$ is a sample space, often containing all the possible outcomes of some experiment.
\item $\mathcal{F}\subseteq 2^{\Omega}$ is a set of events, which forms a $\sigma$-algebra over the set $\Omega$.
\item $P$ is a probability measure defined over the measurable space $(\Omega,\mathcal{F})$.
\end{itemize}
\end{definition}

In a probability space $(\Omega,\mathcal{F},P)$, each element $E\in\mathcal{F}$ is called an {\it event}. An event $E$ occurs {\it almost surely} if $P(E)=1$.

\subsection{Random variables and processes}
\begin{definition}[Random Elements and Random Variables]
Let $(\Omega, \mathcal{F}, P)$ be a probability space, and let $(S,\Sigma_{S})$ be a measurable space. A measurable function $X:\Omega\rightarrow S$, is called a random element in $S$. When $S=\mathbb{R}$, $X$ is called a random variable.
\end{definition}

If $S=\mathbb{R}^n$ for some $n>0$, $X$ is called a {\it random vector}. The concept of a random vector is in a way a generalization of the concept of a random variable. In the discussion that follows we often use the term random variable to refer to both.

A random element $X$ from $(\Omega,\mathcal{F})$ to $(S,\Sigma_S)$ defines (induces) a new measure $P\circ X^{-1}$ on $(S,\Sigma_S)$, which for $B\in\Sigma_S$ satisfies:
\[(P\circ X^{-1})B=P\{\omega\in \Omega:X(\omega)\in B\}.\]
The function $(P\circ X^{-1})$ is again a probability measure on space $(S,\Sigma_S)$, called the {\it probability distribution} (or {\it law}) of $X$. 

Here it is appropriate to briefly introduce the concept of constant random elements. Let $(S,\Sigma_S)$ be a measurable space and for some $x\in S$ let $\delta_{x}$ be a probability measure on $S$ which for $A\in \Sigma_S$ is defined as:
\[\delta_{x}(A)=\begin{cases}
1,~\text{if}~x\in A\\
0,~\text{if}~x\notin A.\\
\end{cases}\]
Then the measure $\delta_{x}$ is called the Dirac measure centred on $x$. Let $X:\Omega\rightarrow S$ be a random element. If for some $x\in S$, $\delta_x$ is the probability distribution of $X$ then $X$ is called a {\it constant (or determininistic) random element}.

Let $X:\Omega\rightarrow S$ and $Y:\Omega\rightarrow S$ be random elements. Then $X$ and $Y$ are {\it equal in distribution} (denoted by $X\stackrel{d}{=}Y$) if and only if for all $B\in\Sigma_S$:
\[(P\circ X^{-1})B=(P\circ Y^{-1})B.\]

The {\it expected value} or {\it expectation} $\mathbb{E}[X]$ of a random variable $X:\Omega\rightarrow\mathbb{R}^n$ is defined as:
\[\mathbb{E}[X]=\int_{\Omega}X\,dP=\int_{\mathbb{R}^n}x\,d(P\circ X^{-1}),\]
if the Lebesgue integral exists.
\begin{remark}
The latter Lebesgue integral sums over the space $\mathbb{R}^n$ according to the measure defined by the probability distribution of $X$. In literature, the integral is almost always written as:
\[\int_{\mathbb{R}^n}x\,(P\circ X^{-1})(dx).\]
In the rest of the text, we will follow this convention.
\end{remark}

For a function $f:\mathbb{R}^n\rightarrow\mathbb{R}^m$ the expectation $\mathbb{E}[f(X)]$ is defined as:
\[\mathbb{E}[f(X)]=\int_{\mathbb{R}^m}x(P\circ(f\circ X)^{-1})(dx).\]
\begin{remark}
Let $X:\Omega\rightarrow \mathbb{R}^n$ be a random variable, and let $A\subseteq\mathbb{R}^n$. We define the special operator $\mathbb{P}$ as follows:
\[\mathbb{P}\{X\in A\}=\mathbb{E}[\mathds{1}_{A}(X)].\]
\end{remark}

Let $X:\Omega\rightarrow\mathbb{R}^n$ be a random variable. The {\it conditional expectation} of $X$ with respect to a $\sigma$-algebra $\mathcal{D}\subseteq\mathcal{F}$, written as $\mathbb{E}[X\mid \mathcal{D}]$, is a $\mathcal{D}$-measurable function which satisfies the following condition:
\[\int_D \mathbb{E}[X\mid \mathcal{D}]\,dP=\int_D X \,dP,~~~\textit{for all }D\in\mathcal{D}.\]
The random variable $\mathbb{E}[X\mid \mathcal{D}]$ which satisfies these equations is unique \cite{billingsley2008probability}.

Consider a random element  $Y:\Omega\rightarrow S$. Given the induced $\sigma$-algebra of $Y$, $\sigma(Y)\subseteq\mathcal{F}$, the conditional expectation of $X$ given $Y$ is a $\sigma(Y)$-measurable function $\mathbb{E}[X\mid Y]$ satisfying:
\[\mathbb{E}[X\mid Y]=\mathbb{E}[X\mid\sigma(Y)].\]

Let $(\Omega,\mathcal{F},\mathbb{P})$ be a probability space. The distinct events $E_1,\ldots,E_n\in\mathcal{F}$ are {\it mutually independent} if and only if:
\[P\left(\bigcap_{i\leq n}E_i\right)=\prod_{i\leq n}P(E_i).\]
For a sequence of $\sigma$-algebras $\mathcal{C}_i\subseteq\mathcal{F}$, where $i\in\{1,\ldots,n\}$, the $\sigma$-algebras are mutually independent if all sequences $A_1,\ldots,A_n$ where $A_i\in \mathcal{C}_i$ are mutually independent.

For $i\in\{1,\ldots,n\}$ let $X_i:\Omega\rightarrow S$ be random elements. The random elements $X_i$ are mutually independent if the sequence of $\sigma$-algebras $\sigma(X_i)$ are mutually independent.

We now possess all the tools to present the following important result regarding the summations of independent and identically distributed (i.i.d.) random variables.
\begin{theorem}[The Weak Law of Large Numbers]\label{theo:wlln}
For $1\leq i\leq n$ let $X_i$ be mutually independent, and identically distributed random variables (i.i.d.'s), with $\mathbb{E}[X_i]=\eta$. Let $S_n=X_1+\ldots+X_n$. Then for any real constant $\varepsilon>0$:
\[\lim_{n\rightarrow\infty}\mathbb{P}\left\{\left\lvert \frac{S_n}{n}-\eta\right\rvert\geq \varepsilon\right\}=0.\]
\end{theorem}

Next, we define stochastic processes.
\begin{definition}[Stochastic Processes]
Let $(\Omega,\mathcal{F},P)$ be a probability space. Let $(S,\Sigma_{S})$ be a measurable space and let $T$ be a totally ordered index set. Let $S^{T}=\{f:T\rightarrow S\}$ be the class of functions from $T$ to $S$, and let $U\subseteq S^{T}$, where $(U,\Sigma_U)$ is a measurable space. A random element $X:\Omega\rightarrow U$ is called a {\it stochastic} or {\it random process} in $S$.
\end{definition}
The elements of $U$ are often called paths or sample paths. However, stochastic processes are more commonly defined as follows.

For $t\in T$ consider the set of all evaluation mappings (functionals) $\pi_t:S^{T}\rightarrow S$ where $\pi_t(f)=f(t)$, and define $X(t)=\pi_t\circ X$. Clearly, for each $t$, $X(t):\Omega\rightarrow S$. Based on~\cite{kallenberg2006foundations} (Lemma 1.4 and Lemma 2.1), since $X$ is measurable and $\pi_t$ are functions, for each $t$, $X(t)$ is also measurable and is a random element in $S$.

Therefore we may also specify an {\it $S$-valued stochastic process} $X$ by a sequence of random elements $\{X(t):t\in T\}$. We write $X(t,\omega)$ when we talk about the value of $X(t)$ for a specific outcome $\omega\in\Omega$. The index set $T$ usually denotes time, and is either discrete ($T=\mathbb{N}$) or continuous ($T=\mathbb{R}_{\geq 0}$).

In practice it is important to switch between the two notions of a stochastic process and employ both intuitions. For Markov processes, they are often specified using the notion of a sequence of random variable, while, when discussing their behaviour they are viewed as function-valued random elements.

\subsection{Markov processes}
Let $T$ be a totally ordered set, called the time domain. Consider the probability space $(\Omega,\mathcal{F},P)$. For a stochastic process we formalize the idea of information known at time $t\in T$ as follows. 

\begin{definition}[Filtration]
Let $\{\mathcal{F}_t\}$ be a sequence of $\sigma$-algebras, where $\mathcal{F}_t\subseteq\mathcal{F}$. If the sequence is increasing i.e., for all $s,t\in T$, $s\leq t$ implies $\mathcal{F}_s\subseteq\mathcal{F}_t$, then $\{\mathcal{F}_t\}$ is called a filtration.
\end{definition}

Let $S$ be a set and let $\{\mathcal{F}_t\}$ be a filtration. An $S$-valued stochastic process $\{X(t):t\in T\}$ is $\{\mathcal{F}_t\}$-adapted if and only if for each $t\in T$, $\mathcal{F}_t$ is the smallest $\sigma$-algebra for which $X(t)$ is $\mathcal{F}_t/\Sigma_S$-measurable i.e., $\mathcal{F}_t=\sigma(X(t))$. For $s\leq t$ the relation $\mathcal{F}_s\subseteq\mathcal{F}_t$ implies that for all $s\leq t$, $X(s)$ is also $\mathcal{F}_t/\Sigma_S$-measurable.

\begin{definition}[Markov Processes]
Let $S$ be a set. An $S$-valued $\{\mathcal{F}_t\}$-adapted stochastic process $\{X(t):t\in T\}$ is a Markov process if for $s,t\geq 0$ and any function $f:S\rightarrow \mathbb{R}^n$:
\[\mathbb{E}[f(X(t+s))\mid\mathcal{F}_t]=\mathbb{E}[f(X(t+s))\mid X(t)].\]
\end{definition}

The above property is called the {\it Markov property}. If $X(t)$ also has the property that for all $x\in S$:
\begin{equation}\label{eq:markov}
\mathbb{E}[f(X(t+s))\mid X(t)=x]=\mathbb{E}[f(X(s))\mid X(0)=x]
\end{equation}
then $X(t)$ is called a {\it time-homogeneous} Markov process.

In what follows we mention some ideas from the theory of Feller semigroups~\cite{kallenberg2006foundations}, which will be used in our proofs. Let $v\in S$ be an initial state, $f:S\rightarrow\mathbb{R}^n$ be a function, and $\{\mathcal{T}_t:t\in \mathbb{R}_{\geq 0}\}$ be a sequence of unary linear operators. A time-homogeneous Markov process $X(t)$ corresponds to the sequence $\{\mathcal{T}_t\}_{t\geq 0}$ if for all $t\geq 0$:
\[\mathcal{T}_t f(v)=\mathbb{E}[f(X(t))\mid X(0)=v].\]

Based on the Markov and time-homogeneity properties, for $s,t\geq 0$ the operators $\{\mathcal{T}_t\}_{t\geq 0}$ satisfy:
\[\mathcal{T}_{s+t}f(v)=\mathcal{T}_t \mathcal{T}_s f(v).\]

As such, $\{\mathcal{T}_t\}_{t\geq 0}$ is called an operator semi-group. It follows that the initial state $v$ and the linear operators $\{\mathcal{T}_t\}_{t\geq 0}$ partially characterize the evolution of the stochastic process $\{X(t):t\in T\}$ throughout time.

Let $S$ be a set and let $T=\mathbb{N}$. Let $\left\{X(t):t\in T\right\}$ be a time-homogeneous Markov process in $S$, $X(t)$ is also called a {\it discrete-time Markov chain} (DTMC). As a special case of $\{\mathcal{T}_t\}$, consider the map $P_t:S\times S\rightarrow[0,1]$, which for $i,j\in S$ is defined as:
\[P_t(i,j)=\mathbb{P}\{X(t)=j\mid X(0)=i\}.\]

In which we have taken $f(X(t))=\mathds{1}_{\{j\}}(X(t))$. The time-homogeneity implies that for $t\geq 0$ and every $i,j\in S$, the map $P_t$ satisfies:
\[P_t=(P_1)^{t},\]
which is called the {\it Chapman-Kolmogorov equation}, and the map $P=P_1$ is called the {\it transition map} or the {\it transition matrix} of $X(t)$. 

Let time $T=\mathbb{R}_{\geq 0}$ be continuous. Let $\{X(t):t\in T\}$ be a time-homogeneous Markov process. For the corresponding linear operator $\{\mathcal{T}_t\}_{t\geq 0}$ the {\it infinitesimal generator} $A$ is a mapping which maps a function $f:S\rightarrow\mathbb{R}$ in its domain to $g:S\rightarrow\mathbb{R}$ (hence $g$ is unique) and satisfies:
\begin{equation}\label{eq:generator}
Af=g=\lim_{t\rightarrow 0^{+}}\frac{\mathcal{T}_t f-f}{t},
\end{equation}

For such a pair $(f,g)\in A$ the following equation always holds~\cite{kurtz1981approximation}:
\begin{equation}\label{eq:generator2}
\mathcal{T}_t f-f=\int_{0}^{t}\mathcal{T}_s g\, ds.
\end{equation}

For a full discussion on the above equation, also see Dynkin's formula~\cite{kallenberg2006foundations}.

\subsection{Martingales and stopping times}
\begin{definition}[Martingales]
Let $\{X(t):t\in T\}$ be an $\{\mathcal{F}_t\}$-adapted stochastic process. Then $X(t)$ is a {\it martingale} if for $s\leq t$:
\[\mathbb{E}[X(t)\mid\mathcal{F}_s]=X(s),\]
almost surely.
\end{definition}

For $T$ discrete, $X(t)$ is a discrete-time martingale if for any $t\in T$ it satisfies:
\[\mathbb{E}[\lvert X(t)\rvert]<\infty~~~,~\mathbb{E}[X(t+1)|X(0),\ldots,X(t)]=X(t)\]
Or equivalently:
\begin{equation}\label{eq:discmart}
\mathbb{E}[X(t+1)-X(t)|X(0),\ldots,X(t)]=0
\end{equation}

In a similar manner, a {\it submartingale} is a process $X(t)$ which for $s>0$ satisfies:
\[\mathbb{E}[X(t+s)\mid\mathcal{F}_t]\geq X(t).\]
with the implication that every martingale is also a submartingale.

Martingales are important due to the fact that despite their generality they satisfy a number of interesting properties. In this text we use one such result called the norm inequality from Doob \cite{kallenberg2006foundations}, which we state in the following form.

\begin{lemma}[Doob's Inequality]
Let $\{X(t):t\in T\}$ be a submartingale taking non-negative values, either in discrete or continuous time. Assume that the process is right continuous with left limits everywhere. Then, for any constant $c>0$,
\[\mathbb{P}\left\{\sup_{0\leq t\leq T}X(t)\geq c\right\}\leq\frac{\mathbb{E}[X(T)]}{c}.\]
\end{lemma}

Based on doob's inequality, for integer $\alpha \geq 1$ the following inequality is derived in~\cite{ethier2009markov} (Proposition 2.2.16):
\begin{equation}\label{eq:doob}
\mathbb{E}\left[\sup_{0\leq t\leq T}X(t)^{\alpha}\right]\leq\left(\frac{\alpha}{\alpha-1}\right)^{\alpha}\mathbb{E}[X(T)^\alpha].
\end{equation}

\begin{definition}[Stopping Times]
Let $\{X(t):t\in T\}$ be a $\{\mathcal{F}_t\}$-adapted process. A random variable $\tau:\Omega\rightarrow T$ is an $\{\mathcal{F}_t\}$-stopping time if for all $t\in T$, the event $\{\tau\leq t\}$ is an element of $\mathcal{F}_t$. 
\end{definition}

In simple terms, the condition means that for a stopping time $\tau$ it must be always possible to determine whether it has occurred by a time $t\in T$ or not, only by referring to the events in $\mathcal{F}_t$. $X(\tau)$ is the state of process $X(t)$ at the random time $\tau$.

Let $X(t)$ be an $S$-valued stochastic process and let $A\subseteq S$. An example of a stopping times is:
\begin{itemize}
\item $\tau_{A}=\min\{t\in T:X(t)\in A\}$ is called the {\it hitting time} of $A$: the first time in which the event $A$ occurs.
\end{itemize}
While the following is not a stopping time:
\begin{itemize}
\item $\tau_A=\max\{t\in T: X(t)\in A\}$: the last time in which the event $A$ occurs.
\end{itemize}

\subsection{Poisson point processes}\label{subsec:poisson}
Before introducing Poisson processes, we draw a link between Bernoulli trials and the so-called {\it Poisson distributions}, using the following approximation theorem.
\begin{theorem}[Poisson Limit Theorem]\label{theo:poisson}
Let $n\in\mathbb{N}$, and for $1\leq i\leq n$ let $Z_i$ be a sequence of i.i.d. random variables, where each $Z_i$ takes value $1$ with probability $p$ and $0$ with probability $1-p$. 

If:
\[n\rightarrow\infty,~p\rightarrow 0, \text{~while~}np\rightarrow\lambda\text{~where~}0\leq\lambda\ll\infty,\]
then
\[\mathbb{P}\left\{ \sum_{i=1}^{n}Z_i=k\right\}\rightarrow e^{-\lambda}\frac{\lambda^k}{k!}.\]
\end{theorem}

We now define counting processes. Consider a set of points (representing events or arrivals) randomly located on the real line $\mathbb{R}_{\geq 0}$, which represents time. We define the process that counts the number of such points in the interval $[0,t]$, $t\in \mathbb{R}_{\geq 0}$. In the following assume $T=\mathbb{R}_{\geq 0}$.

\begin{definition}[Counting Processes]
A $\mathbb{N}$-valued stochastic process $\{N(t):t\in T\}$ is a counting process if:
\begin{enumerate}
\item $\mathbb{P}\{N(0)=0\}=1$
\item For $0\leq s\leq t$, $N(t)-N(s)$ is the number of points in the interval $(s,t]$.
\end{enumerate}
\end{definition}

Any counting process is non-negative, non-decreasing and right-continuous.

\begin{definition}[Independent Increments]
Let $\{N(t):t\in T\}$ be a counting process, and let $t_0,t_1,\ldots,t_n\in T$ be increasing times in which $t_0=0$. We say $N(t)$ has independent increments if and only if for all $i\in\{1,\ldots,n\}$ the random variables:
\[N(t_{i})-N(t_{i-1})\]
are mutually independent. In addition, $N(t)$ is {\it stationary} if for equally distanced $t_0,\ldots,t_n$ the increments $N(t_{i})-N(t_{i-1})$ are equal in distribution.
\end{definition}

\begin{definition}[Time-homogeneous Poisson Processes]
Let $\lambda>0$. A stationary counting process $\{\mathcal{N}(t):t\in T\}$ is called a time-homogeneous Poisson process, or simply a Poisson process with rate or intensity $\lambda$ if it satisfies the following additional properties:
\begin{itemize}
\item $\mathcal{N}(t)$ has independent increments.
\item Almost surely, in an infinitesimal time interval $dt$ at most one point occurs with probability $\lambda dt$.
\end{itemize}
\end{definition}

For a Poisson process, the number of observations over the interval $(s,t]$ is discrete and is distributed according to a Poisson distribution:

\[\mathbb{P}\{\mathcal{N}(t)-\mathcal{N}(s)=k\}=\frac{(\lambda(t-s))^k}{k!}e^{-\lambda(t-s)}.\]

Intuitively speaking, since a step in continuous time satisfies $dt\rightarrow 0$, and in intervals of size $dt$ almost surely only 0 or 1 points may occur, one can think of the process in any interval $(s,t]$ as an infinite sequence of Bernoulli trials and then apply the Poisson limit theorem to derive the above probability.

The Poisson process $\{\mathcal{Y}(t):t\in T\}$ with $\lambda=1$ is called the {\it unit Poisson process}. Let $\mathcal{N}(t)$ be a Poisson process  with $\Lambda(t)=\mathbb{E}[\mathcal{N}(t)]$ and $\Lambda(0)=0$. The following relation holds between $\mathcal{N}(t)$ and the unit Poisson process:
\begin{equation}\label{eq:unitpoisson}
\mathcal{N}(t)=\mathcal{Y}(\Lambda(t)).
\end{equation}

For a unit Poisson process we have $\mathbb{E}[\mathcal{Y}(t)]=t$. A {\it compensated unit Poisson process} $\tilde{\mathcal{Y}}(t)$ is defined as:
\[\tilde{\mathcal{Y}}(t)=\mathcal{Y}(t)-t,\]
For which for all $t\in T$:
\[\mathbb{E}[\tilde{\mathcal{Y}}(t)]=0.\]
Based on equation~(\ref{eq:discmart}) and due to the independent increment property, $\tilde{\mathcal{Y}}(t)$ is a martingale, since for any $s<t$ with $s,t\in T$:
\begin{flalign*}
\mathbb{E}\left[\tilde{\mathcal{Y}}(t)-\tilde{\mathcal{Y}}(s)\mid \mathcal{F}_s\right]&=\mathbb{E}[\mathcal{Y}(t)-\mathcal{Y}(s)\mid \mathcal{F}_s]-(t-s)\\
&=\mathbb{E}[\mathcal{Y}(t-s)\mid \mathcal{F}_s]-(t-s)=0.
\end{flalign*}

\subsection{Continuity and convergence}\label{subsec:con}
In this section we quickly review a number of useful results in the field of real analysis, as the proofs we give later rely heavily on them.

\begin{definition}
[Right Continuous Functions with Left Limits]
Let $f:\mathbb{R}\rightarrow\mathbb{R}^n$ be a function. Then $f$ is a right continuous function with left limits (c\`{a}dl\`{a}g\footnote{continue \`{a} droite, limite \`{a} gauche}) if and only if for every $x\in\mathbb{R}$:
\begin{itemize}
\item $f(x-)\defeq\lim_{a\rightarrow x^{-}}f(a)$ exists, and
\item $f(x+)\defeq\lim_{a\rightarrow x^{+}}f(a)$ exists and $f(x+)=f(x)$.
\end{itemize}
\end{definition}

\begin{definition}
[Lipschitz Continuity]
A function $f:\mathbb{R}^n\rightarrow\mathbb{R}^n$ is (globally) Lipschitz continuous on $\mathbb{R}^n$ if and only if:
\[\exists L\in\mathbb{R}.\forall x_1,x_2\in\mathbb{R}^n.\,\lvert f(x_2)-f(x_1)\rvert\leq L\,\lvert x_2-x_1\rvert\]
\end{definition}
For a Lipschitz continuous function $f$, we call the constant $L$ found above the {\it Lipschitz constant}. If a function $f$ has bounded first derivatives everywhere on $\mathbb{R}^n$, it is guaranteed to be Lipschitz continuous.

\begin{theorem}[Picard-Lindel\"of Theorem]\label{the:piclin}
Let $f$ be a Lipschtiz continuous function in $\mathbb{R}^n$, and consider the following ordinary differential equation:
\[\frac{d}{dt}x(t)=f(x(t))~,~~x(t_0)=x_0,\]
then for some $\epsilon>0$, there exists a unique solution $x(t)$ to the initial value problem on the interval $[t_0-\epsilon,t_0+\epsilon]$.
\end{theorem}

\begin{lemma}[Gr\"onwall's Inequality]\label{lem:gronwall}
Let $f$ be a function that is bounded and integrable on the interval $[0,T]$, if 
\[f(T)\leq C+D\int_0^{T}f(t)dt\]
then 
\[f(T)\leq Ce^{DT}\]
\end{lemma}

\begin{definition}[Uniform Convergence]
Let $\left\{f_n\right\}$ be a sequence of functions on set $S$. We say the sequence converges uniformly to function $f$ if:
\[\lim_{n\rightarrow\infty}\sup_{x}\lvert f_n(x)-f(x)\rvert=0,\]
that is, the speed of convergence of $f_n(x)$ to $f(x)$ does not depend on $x$.
\end{definition}
If the sequence of functions $\{f_n\}$ are continuous and converge uniformly to $f$, then the limiting function $f$ is continuous as well.

\begin{theorem}[Lebesgue Dominated Convergence Theorem]\label{the:lebesgue}
Let $\left\{f_n\right\}$ be a sequence of measurable functions on set $S$, which converge to measurable function $f$. Let $g$ be an integrable function such that for all $n$ and for all $x\in S$: $\lvert f_n(x)\rvert\leq \lvert g(x)\rvert$. Then $f$ is integrable and:
\[\lim_{n\rightarrow\infty}\int_{S}\lvert f_n-f\rvert=0.\]
\end{theorem}

Finally, we give several notions of convergence for random variable, and quickly overview their relations. 
\begin{definition}
[Convergence in $L^p$, in Probability and in Distribution] Let $\{X_n\}$ be a sequence of random variables, and $X$ be a random variable. Then for $p\in\mathbb{N}_{>0}$, $\{X_n\}$ converges to $X$ in $L^p$ or $p$-th moment if:
\[\lim_{n\rightarrow\infty}\mathbb{E}\left[\lvert X_n-X\rvert^p\right]=0,\]
we say $\{X_n\}$ converges in probability to $X$ if for any $\epsilon>0$:
\[\lim_{n\rightarrow\infty}\mathbb{P}\left\{\left\lvert X_n-X\right\rvert\right\}=0,\]
and we say $\{X_n\}$ converges in distribution to $X$ if for all $x\subset\mathbb{R}^m$:
\[\lim_{n\rightarrow\infty} (P\circ X_n)^{-1}x=(P\circ X)^{-1}x.\]
\end{definition}
Convergence in $L^2$ is also called {\it convergence in mean-square}. If possible, proving convergence in mean-square is very useful, since it implies convergence in probability, which in turn implies convergence in distribution. In this sense convergence in distribution is also often referred to as {\it weak convergence}.

\section{Mean field approximation}\label{sec:method}
In this section, we present the stochastic model of a system and its mean field approximations. For the most part, our notation agrees with~\cite{benaim2008class}. A list of objects appearing in the mathematical discussions that follow are given in Table~\ref{tbl:objects}.

\begin{table}
\caption{Table of objects and their short description.}\label{tbl:objects}
\centering
	\def\arraystretch{1.2}
\begin{tabular}{ l  c  p{11cm} }
	\toprule
	$T=\mathbb{N}$ &  & Points corresponding to local time-slots\\
	$T_G\subseteq\mathbb{Q}_{\geq 0}$ &  & Points on the real line corresponding to global time-slots\\
	$D\in\mathbb{N}_{\geq 1}$ &  & Time resolution = number of global time-slots in a unit interval\\
	$\epsilon=\frac{1}{D}$ &  & Length of a global time-slot\\
	$N\in\mathbb{N}_{\geq 1}$ &  & System size = number of agents\\
	$\mathcal{S}=\{1,\ldots,I\}$ &  & State space of agents, with $I\in\mathbb{N}$ states\\	
	$\left\{X_i^{(N)}(t):t\in T\right\}$ &  & Process corresponding to agent $i$, with $i\in\{1,\ldots,N\}$\\
	$K_i$ &  & Transition map of $X_i^{(N)}(t)$\\
	$\left\{\hat{X}_i^{(N)}(t):t\in T_G\right\}$ &  & Modified process corresponding to agent $i$\\
	$\hat{K}_i$ &  & Transition map of $\hat{X}_i^{(N)}(t)$\\
	$\left\{Y^{(N)}(t):t\in T_G\right\}$ & & Process for the system of $N$ agents, on $\mathcal{S}^N$\\
	$\mathcal{K}^{(N)}$ & & Transition map of $Y^{(N)}(t)$\\
	$\Delta$ & & Set of occupancy measures\\
	$\left\{M^{(N)}(t):t\in T_G\right\}$ & & Normalised population process on $\Delta^{(N)}\subset\Delta$\\
	$\hat{\mathcal{K}}^{(N)}$ & & Transition map of $M^{(N)}(t)$\\
	$P_1^{(N)}$ & & Transition map of the agent model $\left\{\left(\hat{X}_1^{(N)},M^{(N)}(t)\right):t\in T_G\right\}$\\
	$P_{s,s'}^{(N)}$ & & Agent transition map, with $s,s'\in\mathcal{S}$\\
	$Q_{s,s'}^{(N)}$ & & Infinitesimal agent transition map, with $s,s'\in\mathcal{S}$\\
	$\left\{\bar{M}^{(N)}(t):t\in\mathbb{R}_{\geq 0}\right\}$ & & Normalised population process with continuous paths\\
	$\left\{W^{(N)}(t):t\in T_G\right\}$ & & Object (agent) state-change frequency in interval $[0,t]$\\
	$\hat{F}^{(N)}$ & & Expected instantaneous change in system state\\
	$F^{(N)}$ & & Drift of the normalized population process\\
	$\Phi\subseteq\{g:\mathbb{R}_{\geq 0}\rightarrow\Delta\}$ & & Set of deterministic approximations\\
	$F^{*}$ & & The limit of the sequence of drifts $\left\{F^{(N)}\right\}$\\
	$\rho_N$ & & The probability measure induced by $Y^{(N)}(t)$\\
	$\varepsilon_N$ & & The empirical measure derived from $Y^{(N)}(t)$\\
	$\tilde{F}^{(N)}_{s,s'}$ & & The Poisson mean of intensity from $s$ to $s'$ with $s,s'\in\mathcal{S}$\\
	$\tilde{F}^{(N)}$ & & The mean drift of the normalised population process\\
	\bottomrule
\end{tabular}
\end{table}

\subsection{Agent processes and the clock independence assumption}\label{subsec:independentevent}
Let the set $T=\mathbb{N}$ be discrete and let parameter $N\in\mathbb{N}_{\geq 1}$ be the {\it system size}. The elements of $T$ are called time-slots. Let $\mathcal{S}=\{1,\ldots,I\}$ be a finite set of states. For $i\in\{1,\ldots,N\}$, let $\left\{X^{(N)}_i(t):t\in T\right\}$ be $\mathcal{S}$-valued discrete-time time-homogeneous Markov chains (DTMCs). Each stochastic process $X^{(N)}_i(t)$ describes the behaviour of agent $i$ in the system with $N$ agents.

Take each process $X^{(N)}_i(t)$ to be described by a transition map $K_i:\mathcal{S}^N\times\mathcal{S}\rightarrow [0,1]$. In each time-slot (indexed by members of $T$), the process chooses the next state $s\in\mathcal{S}$ with probability $K_i(\vec{v},s)$, where the vector of states $\vec{v}\in\mathcal{S}^N$ is the state of the entire system (including agent $i$'s current state).

There are generally two ways in which we can relate the time-slots across the processes in such a system:
\begin{enumerate}
\item \emph{The time-slots are fully synchronized} and the $N$ processes simultaneously update their states.
\item \emph{Processes have independent time-slots, which occur at the same rate over sufficiently long intervals of time}.\label{item:iea}
\end{enumerate}

The two approaches may lead to systems with different behaviours (see remark below). For a discussion on the approximation of systems with simultaneous update (or synchronous DTMCs) refer to~\cite{bortolussi2013continuous,le2007generic}. Our discussion is about systems with independent time-slots, since this assumption allows us to embed the discrete-time description of agents' behaviours in a continuous-time setting. 

Formally, the type~\ref{item:iea} behaviour can be stated as follows. For two processes $i$ and $i'$ where $i\neq i'$, if process $i$ does a transition in an instant of time then process $i'$ almost never does a transition simultaneously. Here we say that these systems satisfy the {\it clock independence assumption}.

Technically, we enforce the clock independence assumption through scaling the duration of time-slots and modifying agent transition probabilities as follows. Let $D\in\mathbb{N}_{\geq 1}$ be the time resolution, and let $\epsilon\in\mathbb{Q}_{\geq 0}$ be a positive rational number (a probability) defined as $\epsilon=\frac{1}{D}$. Let $T_G\subseteq\mathbb{Q}_{\geq 0}$ be the countable set:
\[T_G=\left\{0,\epsilon,2\epsilon,\ldots\right\}.\]
We call the set $T_G$ the {\it system} or {\it global time}, as opposed to the {\it agent} or {\it local time} $T$. Next, let the probability of an agent doing a transition in a time-slot be $\epsilon$. In this new setting, for $1\leq i\leq N$ define stochastic processes $\left\{\hat{X}^{(N)}_i(t):t\in T_G\right\}$, each with transition maps $\hat{K}_i:\mathcal{S}^N\times\mathcal{S}\rightarrow [0,1]$, such that for all $\vec{v}\in\mathcal{S}^N$ and $s\in\mathcal{S}$:
\begin{equation*}
\hat{K}_i(\vec{v},s)=
\begin{cases}
\epsilon\, K_i(\vec{v},s)\quad\quad\quad\quad\quad~\text{if}~s\neq \vec{v}_i,\\
(1-\epsilon)+\epsilon\, K_i(\vec{v},s)~~~~~~\text{if}~s=\vec{v}_i.\\
\end{cases}
\end{equation*}

In the new setting, let $E$ be the event that agent $i$ does a transition in a time-slot, and $E'$ be the event that agent $i'\neq i$ does a transition exactly in the same time-slot in $T_G$. Then by independence of agent transition maps:
\[\mathbb{P}\{E'|E\}=\mathbb{P}\{E'\}=\epsilon.\]
Observe that the clock independence assumption is satisfied as $D\rightarrow\infty$ (i.e., $\epsilon\rightarrow0$):
\[\lim_{D\rightarrow\infty}\mathbb{P}\{E'|E\}=0.\]

\begin{remark}
A condition under which systems with simultaneous updates and the clock independence assumption show divergent behaviours is when agent processes are not aperiodic. In such cases the system satisfying the clock independence assumption becomes aperiodic, while the system with simultaneous updates is not aperiodic.
\end{remark}

Let $\left\{Y^{(N)}(t):t\in T_G\right\}$ be a stochastic process with states $Y^{(N)}(t)=\left(\hat{X}^{(N)}_1(t),\ldots,\hat{X}^{(N)}_N(t)\right)$. The process $Y^{(N)}(t)$ represents the behaviour of the entire system, and is a time-homogeneous discrete-time Markov process with transition map $\mathcal{K}^{(N)}:\mathcal{S}^{N}\times\mathcal{S}^{N}\rightarrow [0,1]$ in which for $\vec{v},\vec{v}\,'\in \mathcal{S}^{N}$: 
\[\mathcal{K}^{(N)}(\vec{v},\vec{v}\,')=\prod_{i=1}^{N}\hat{K}_i(\vec{v},\vec{v}\,'_i)\,.\]

\subsection{Mean field interaction models}\label{subsec:popprocess}
In this part we define mean field interaction models \cite{benaim2008class}, which comprise the class of processes $Y^{(N)}(t)$ for which we find the mean field approximations.

Let $\pi:\{1,\ldots,N\}\rightarrow\{1,\ldots,N\}$ be a bijection. The function $\pi$ is called a permutation over the set $\{1,\ldots,N\}$. Additionally, for a vector $\vec{v}=(s_1,\ldots,s_N)$ define $\pi(\vec{v})$ as:
\[\pi(\vec{v})=\big(s_{\pi(1)},\ldots,s_{\pi(N)}\big)\,.\]

\begin{definition}[Mean Field Interaction Models~\cite{benaim2008class}]
Let $Y^{(N)}(t)$ be the process defined earlier, and let $\pi$ be any permutation over the set $\{1,\ldots,N\}$. If for all $\vec{v},\vec{v}\,'\in\mathcal{S}^N$:
\[\mathcal{K}^{(N)}(\vec{v},\vec{v}\,')=\mathcal{K}^{(N)}(\pi(\vec{v}),\pi(\vec{v}\,'))\]
holds, $Y^{(N)}(t)$ is called a mean field interaction model with $N$ objects.
\end{definition}

It follows from the above definition that entries in $\mathcal{K}^{(N)}$ may depend on the number of agents in each state, but not on the state of a certain agent. Let $\Delta=\big\{\vec{m}\in\mathbb{R}^{I}:\sum_{s\in\mathcal{S}}\vec{m}_s=1\wedge\forall s.\vec{m}_s\geq 0\big\}$ be a set of vectors, which we call the set of {\it occupancy measures}. For a system of size $N$, take the countable subset $\Delta^{(N)}=\big\{\vec{m}\in\mathbb{R}^{I}:\sum_{s\in\mathcal{S}}\vec{m}_s=1\wedge\forall s.(\vec{m}_s\geq 0\wedge N\vec{m}_s\in\mathbb{N})\big\}$. The set $\Delta^{(N)}\subset \Delta$ is an alternative representation of the state space of the system, in which for $\vec{m}\in\Delta^{(N)}$ and $i\in \mathcal{S}$ the value $\vec{m}_i$ expresses the proportion of agents that are in state $i$. For a mean field interaction model we define the {\it normalized population process} $\left\{M^{(N)}(t):t\in T_G\right\}$ on $\Delta^{(N)}$ such that for $s\in\mathcal{S}$:
\begin{equation}\label{eq:normal}
M^{(N)}_s(t)=\frac{1}{N}\sum_{1\leq n\leq N}\mathds{1}\left(\hat{X}^{(N)}_n(t)=s\right),
\end{equation}
where $\mathds{1}$ is an indicator function\footnote{Indicator functions are defined in Section~\ref{subsec:measure}.}\footnote{For any system of size $N$, we should be careful about the occupancy measures that refer to invalid/meaningless states, since the set $\Delta$ is uncountable.}. Using the fact that $Y^{(N)}(t)$ is a mean field interaction model, it is possible to move back and forth between processes $Y^{(N)}(t)$ and $M^{(N)}(t)$. In the following discussion we illustrate this fact.

Let $\vec{m},\vec{m}'\in\Delta^{(N)}$ be occupancy measures. Let $\vec{v}\in\mathcal{S}^N$ be a vector defined as:
\[\vec{v}=(\overbrace{1,\ldots,1}^{N\vec{m}_1},\ldots,\overbrace{I,\ldots,I}^{N\vec{m}_I})\,,\]
where the notation $\overbrace{s,\ldots,s}^{\ell}$ shows a sequence of length $\ell$ in which all the elements have value $s$. Thus the vector $\vec{v}$ consists of $N\vec{m}_1$ elements with value 1,  $N\vec{m}_2$ elements with value 2, and so forth. Consider the set $V'$ where:
\[V'=\Big\{\vec{v}\,'\in \mathcal{S}^N\mid \exists\, \pi.(\vec{v}\,'=\pi(\overbrace{1,\ldots,1}^{N\vec{m}\,'_1},\ldots,\overbrace{I,\ldots,I}^{N\vec{m}\,'_I}))\Big\}\,.\]
The set $V'$ is the set of all possible configurations of the process $Y^{(N)}(t)$ which translate to $\vec{m}\,'$ after normalization.
For process $M^{(N)}(t)$ we define the new transition map $\hat{\mathcal{K}}^{(N)}:\Delta^{(N)}\times\Delta^{(N)}\rightarrow [0,1]$:
\begin{equation*}
\hat{\mathcal{K}}^{(N)}(\vec{m},\vec{m}\,')=\sum_{\vec{v}\,'\in V'}\mathcal{K}^{(N)}(\vec{v},\vec{v}\,')
\end{equation*}
Similarly, we define the process $\left\{(\hat{X}^{(N)}_1(t),M^{(N)}(t)):t\in T_G\right\}$ which models the behaviour of an agent in the context of a mean field interaction model, and call it the {\it agent model}. Let $\vec{m},\vec{m}'\in\Delta^{(N)}$ be occupancy measures and $s,s'\in\mathcal{S}$ be states, define the vector:
\[\vec{v}=(s,\overbrace{1,\ldots,1}^{N\vec{m}_1},\ldots,\overbrace{s\phantom{'},\ldots,s\phantom{'}}^{N\vec{m}_s-1},\ldots,\overbrace{I,\ldots,I}^{N\vec{m}_I}).\]
Consider the set $V'$ which consists of all the permutations of the vector:
\[V'=\Big\{\vec{v}\,'\in \mathcal{S}^N\mid (\vec{v}\,'_1=s')\wedge\exists\, \pi.(\vec{v}\,'=\pi(\overbrace{1,\ldots,1}^{N\vec{m}\,'_1},\ldots,\overbrace{I,\ldots,I}^{N\vec{m}\,'_I}))\Big\}.\]

The transition map of the agent model is the function $P^{(N)}_1:\mathcal{S}\times\Delta^{(N)}\times\mathcal{S}\times\Delta^{(N)}\rightarrow [0,1]$ defined as follows:
\begin{equation*}
P^{(N)}_1(s,\vec{m},s',\vec{m}\,')=
\begin{cases}
\sum_{\vec{v}\,'\in V'}\mathcal{K}^{(N)}(\vec{v},\vec{v}\,')\quad\text{if}~N\vec{m}\,'_{s'}\geq 1 \text{~and~}N\vec{m}_s\geq 1,\\
0\quad\quad\quad\quad\quad\quad\quad\quad\text{otherwise.}\\
\end{cases}
\end{equation*}

Based on $P^{(N)}_1$, for $s,s'\in\mathcal{S}$ we also define the probability:
\[P^{(N)}_{s,s'}(\vec{m})=\sum_{\vec{m}\,'\in\Delta}P^{(N)}_1(s,\vec{m},s',\vec{m}'),\]
that an agent moves from state $s$ to state $s'$, in the context $\vec{m}\in\Delta^{(N)}$.

For each $s,s'\in\mathcal{S}$ where $s\neq s'$, the expected proportion of the agents that are in state $s$ at time $t$ and move to state $s'$ over a unit time interval $[t,t+1]$ is:
\[\sum_{i=0}^{D}i\,\binom{D}{i}\left(P^{(N)}_{s,s'}(\vec{m})\right)^i\left(1-P^{(N)}_{s,s'}(\vec{m})\right)^{D-i}=D\,P^{(N)}_{s,s'}(\vec{m}).\]
By taking the clock independence assumption into account, for $s\neq s'$ define the functions $Q^{(N)}_{s,s'}:\Delta^{(N)}\rightarrow\mathbb{R}_{\geq 0}$ which for $\vec{m}\in\Delta^{(N)}$ satisfy:
\[Q^{(N)}_{s,s'}(\vec{m})=\lim_{D\rightarrow\infty}D\,P^{(N)}_{s,s'}(\vec{m}).\]
Note that due to the construction of the probabilities $P^{(N)}_{s,s'}(\vec{m})$, for $s\neq s'$ the limit always exists (i.e., it is in the interval $[0,1]$). The mapping $Q^{(N)}$ is called the infinitesimal agent transition map, which can be interpreted as a transition rate matrix, meaning that for an agent in state $s$ the time until it moves to state $s'$ converges to an exponentially distributed random variable with mean: $\frac{1}{Q^{(N)}_{s,s'}}$.

The time instants in the set $T_G$ are discrete. However,  as we see later in the approximation it is necessary to observe the population process $\left\{M^{(N)}(t):t\in T_G\right\}$ at continuous times $t\in\mathbb{R}_{\geq 0}$. Based on $M^{(N)}(t)$ we define a new stochastic process $\left\{\bar{M}^{(N)}(t):t\in\mathbb{R}_{\geq 0}\right\}$ with new sample paths which are right-continuous functions with left limits (c\`{a}dl\`{a}gs). For $t\in\mathbb{R}_{\geq 0}$ the process $\bar{M}^{(N)}(t)$ satisfies:
\[\bar{M}^{(N)}(t)=M^{(N)}\left(\epsilon\left\lfloor D t\right\rfloor\right).\]

To illustrate how one derives the mean field approximation of a system, we use a running example in this and the following sections. In the example of this section we start from an informal specification of an agent behaviour, and derive the transition map of the corresponding normalized population process.

\begin{example}\label{ex:interf1}
Consider a network of $N$ nodes (agents) operating on a single shared channel. The network is saturated, meaning that all the nodes always have messages to transmit. A node can be in one of the states $\mathcal{S}=\{1,2\}$. In state $1$ a node is waiting, and with probability $p_1$ decides to transmit a message. All communications start with the transmission of a message, and a successful communication is then marked by the receipt of an acknowledgement, whereas a failed communication ends in a timeout. Both cases occur in the space of a single time-slot. If the communication succeeds, the node will remain in state $1$ and wait to transmit the next message and if it fails, 
the node moves to state $2$ in which it tries retransmitting the message. A node in state $2$ retransmits the message with probability $p_2$. The node then essentially behaves in the same way as in state $1$.

The probability of success depends on the number of nodes currently using the channel, as follows. If $n\in\mathbb{N}$ nodes are using the channel, then the success probability of each participating communication is:
\begin{equation}\label{eq:ex1ps}
p_s(n)=2^{-n}\,,
\end{equation}
that is, the channel degrades in quality exponentially as the number of active nodes increases. A diagram representing the behaviour of each node is given in Figure~\ref{fig:markovchain}.

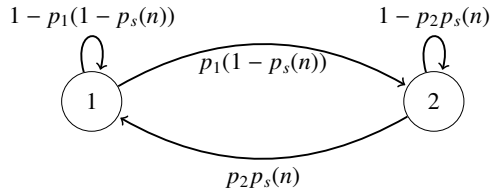
\begin{figure}[h]
\centering
\scalebox{0.9}{
\begin{tikzpicture}[shorten >=1pt,node distance=5cm,on grid,auto]
   \node[state] (q_0) {$1$};
   \node[state] (q_1) [right=of q_0] {$2$};   
    \path[->,thick]
	(q_0) edge [swap,bend left] node {$p_1(1-p_s(n))$} (q_1)
		  edge [loop above] node {$1-p_1(1-p_s(n))$} (q_0)    
    (q_1) edge [bend left] node {$p_2 p_s(n)$} (q_0)
    	  edge [loop above] node {$1-p_2 p_s(n)$} (q_1)
    ;
\end{tikzpicture}
}
\caption{The behaviour of a node in Example~\ref{ex:interf1}. The number of transmitting nodes is $n$.} \label{fig:markovchain}
\end{figure}

Let $\vec{v}\in\mathcal{S}^N$ be the state of the network, with $n_1$ nodes in state $1$ and $n_2=N-n_1$ nodes in state $2$. Let $tr_1(\vec{v})$ be the total number of nodes that are in state $1$ which decide to transmit a message in system state $\vec{v}$, then $tr_1(\vec{v})$ is a binomial random variable with distribution $B(n_1,p_1)$. In a similar fashion $tr_2(\vec{v})$, the total number of nodes in state $2$ which decide to transmit a message in system state $\vec{v}$ is $B(n_2,p_2)$ distributed. Then a communication in state $\vec{v}$ will succeed with probability $p_s\left(tr_1(\vec{v})+tr_2(\vec{v})\right)$. 

Let $s\in\mathcal{S}$ be the next state of a node $i$ ($1\leq i\leq N$), based on the description above the transition matrix for this node is:
\[K_{i}(\vec{v},s)=\begin{pmatrix}
1-p_1+p_1 p_s\left(tr_1(\vec{v})+tr_2(\vec{v})\right)~~ & ~~p_1 - p_1 p_s\left(tr_1(\vec{v})+tr_2(\vec{v})\right) \\
p_2 p_s\left(tr_1(\vec{v})+tr_2(\vec{v})\right) & 1-p_2 p_s\left(tr_1(\vec{v})+tr_2(\vec{v})\right) \\
\end{pmatrix}\]
where the row is determined by the element $\vec{v}_i$ (current state) and the column by $s$.

We use the clock independence assumption to compose the population process. To extend this assumption to the description of our radio network, we implicitly assume that the duration of message transmission is exponentially distributed, i.e., since the transitions are memoryless, the sojourn time of individuals in states is exponentially distributed. The modified transition matrix for node $i$ is:
\[\hat{K}_{i}(\vec{v},s)=\begin{pmatrix}
1-p_1\epsilon\,(1- p_s\left(tr_1(\vec{v})+tr_2(\vec{v})\right))~~ & ~~p_1\epsilon\,(1 - p_s\left(tr_1(\vec{v})+tr_2(\vec{v})\right)) \\
p_2\epsilon\, p_s\left(tr_1(\vec{v})+tr_2(\vec{v})\right) & 1-p_2\epsilon\, p_s\left(tr_1(\vec{v})+tr_2(\vec{v})\right) \\
\end{pmatrix}\]
In which the probability of success and failure have been scaled by a factor $\epsilon$. The composed system is a mean field interaction model. This is due to the definition of functions $tr_1$ and $tr_2$, which do not depend on states of specific nodes, but rather on the aggregate number of nodes in states $1$ and $2$.

Consider the normalized population model with occupancy measures $\Delta=\big\{\vec{m}\in\mathbb{R}^{2}:\sum_{i}\vec{m}_i=1\wedge\forall i.\vec{m}_i\geq 0\big\}$ and the corresponding subset $\Delta^{(N)}$. Let $\vec{m}\in\Delta^{(N)}$; then using (\ref{eq:normal}) when the system is in state $\vec{m}$ the total number of communicating agents is $X_1+X_2$ where $X_1 \sim B(N\vec{m}_1 ,p_1)$ and $X_2 \sim B(N\vec{m}_2 ,p_2)$. For an agent in state $s$ the rate of moving to an state $s'\neq s$ is given by:
\begin{equation}\label{eq:ex1result}
\begin{cases}Q^{(N)}_{1,2}(\vec{m})=
\mathbb{E}\left[p_1(1-p_s(X_1+X_2))\right]\,, \\
Q^{(N)}_{2,1}(\vec{m})=\mathbb{E}\left[p_2 p_s(X_1+X_2)\right]\,.
\end{cases}
\end{equation}

\end{example}

The map $Q^{(N)}$ derived here defines the stochastic behaviour of an agent in the population process. In the sections that follow, we use this map to derive the mean field approximation of the population process.

\subsection{Drift and the time evolution of the process $\bar{M}^{(N)}(t)$}\label{subsec:drift}
In this section we define the drift as a way to characterize the behaviour of the normalized population process $\bar{M}^{(N)}(t)$ in its first moment. This provides the basis for defining the mean drift, which is given in Section~\ref{sec:propagation}.

Define $W^{(N)}_{s,s'}(t)$ as the random number of objects which do a transition from state $s$ to state $s'$ in the system at time $t\in T_G$, i.e., 
\begin{equation}\label{eq:counter}
W^{(N)}_{s,s'}(t+\epsilon)=\sum_{k=1}^{N}\mathds{1}\left\{\hat{X}^{(N)}_k(t)=s,\hat{X}^{(N)}_k(t+\epsilon)=s'\right\}.
\end{equation}
The instantaneous changes of the system $M^{(N)}(t)$ can be tracked by the following random process:
\[M^{(N)}(t+\epsilon)-M^{(N)}(t)=\sum_{s,s'\in\mathcal{S},s\neq s'}\frac{W^{(N)}_{s,s'}(t+\epsilon)}{N}(\vec{e}_{s'}-\vec{e}_s)\]
where $\vec{e}_s$ is a unit vector of dimension $I$ with value 1 in position $s$. Then the expected value of the instantaneous change is the function $\hat{F}^{(N)}:\Delta^{(N)}\rightarrow\mathbb{R}^{I}$ where:
\begin{flalign*}
\hat{F}^{(N)}(\vec{m})&=\mathbb{E}\left[M^{(N)}(t+\epsilon)-M^{(N)}(t)\mid M^{(N)}(t)=\vec{m}\right]=\sum_{s,s'\in\mathcal{S}}\vec{m}_s P^{(N)}_{s,s'}(\vec{m})(\vec{e}_{s'}-\vec{e}_s).
\end{flalign*}

The {\it drift} is the function $F^{(N)}:\Delta^{(N)}\rightarrow \mathbb{R}^{I}$ defined as:
\[F^{(N)}(\vec{m})=\lim_{D\rightarrow\infty}D\hat{F}^{(N)}(\vec{m})=\sum_{s,s'\in\mathcal{S},s\neq s'}\vec{m}_s Q^{(N)}_{s,s'}(\vec{m})(\vec{e}_{s'}-\vec{e}_s).\]
In the above formula, we may use $F_{s,s'}^{(N)}(\vec{m})$ to represent the summand:
\[\vec{m}_s\, Q^{(N)}_{s,s'}(\vec{m}),\]
which we call the {\it intensity} of transitions from $s$ to $s'$.

Below, we present properties that the drifts may or may not satisfy, and are of interest in our discussion.

\vspace{2mm}

{\bf Smoothness}: For all $N\geq 1$, there exist Lipschitz continuous functions $\bar{F}^{(N)}:\Delta\rightarrow\mathbb{R}^I$ which for all $\vec{m}\in\Delta^{(N)}$ satisfy: $\bar{F}^{(N)}(\vec{m})=F^{(N)}(\vec{m})$.
\vspace{2mm}

{\bf Boundedness}: For all $N\geq 1$, $F^{(N)}$ are bounded on $\Delta^{(N)}$.

\vspace{2mm}

{\bf Limit existence}: Assuming smoothness, the sequence of drifts $\{F^{(N)}\}$ converges uniformly to a bounded function $F^{*}:\Delta\rightarrow\mathbb{R}^I$.

\vspace{2mm}

In the literature, the single term {\it density dependence} is often used to refer to {\bf boundedness} and {\bf limit existence}~\cite{kurtz1978strong}.

\begin{remark}
The {\bf smoothness} assumption appears in more or less the same shape throughout the literature. However our version is slightly more restrictive, since it allows us to skip some discussions regarding topological spaces. Moreover, in the contexts where it is clear that {\bf smoothness} holds, we overload the name $F^{(N)}$ to refer to the function $\bar{F}^{(N)}$ instead.
\end{remark}

Essentially, the drift extends the vector representing the expected instantaneous changes into the unit time interval. Using drift, one can express how the expected value $\mathbb{E}[\bar{M}^{(N)}(t)]$ will evolve over time, a fact which we formally express through the following proposition.

\begin{proposition}\label{pro:driftgen}
For the process $\bar{M}^{(N)}(t)$, and its drift $F^{(N)}$ the following equation holds:
\[
\mathbb{E}\left[\bar{M}^{(N)}(t)\mid \bar{M}^{(N)}(0)\right]-\bar{M}^{(N)}(0)=\int_{0}^{t}\mathbb{E}\left[F^{(N)}(\bar{M}^{(N)}(s))\mid \bar{M}^{N}(0)\right]ds.
\]
\end{proposition}

In its differential form, the equation above suggests that the expected trajectory of the process $\bar{M}^{(N)}(t)$ is a solution of the following system of ordinary differential equations:
\begin{equation}\label{eq:ode2}
\frac{d}{d t}\mathbb{E}\left[\bar{M}^{(N)}(t)\right]=\mathbb{E}\left[F^{(N)}(\bar{M}^{(N)}(t))\right]
\end{equation}
with the initial value $\bar{M}^{(N)}(0)$. In practice the term $\mathbb{E}\left[F^{(N)}(\bar{M}^{(N)}(t))\right]$ is difficult to describe. Indeed, one of the reasons why we are interested in the mean field approximation is to avoid the exact calculation of the distribution of the random process $\bar{M}^{(N)}(t)$. In Section~\ref{sec:propagation} we propose a way to approximate the right hand side of equation (\ref{eq:ode2}) by expressing it in terms of $\mathbb{E}\left[\bar{M}^{(N)}(t)\right]$, without explicitly giving the error bounds.

In the following example, we continue towards a mean field approximation for the system defined in example~\ref{ex:interf1}.

\begin{example}\label{ex:ex2}
We use the maps $Q^{(N)}_{s,s'}$ given by (\ref{eq:ex1result}) to derive the drift of the system described in Example \ref{ex:interf1}. A simple substitution gives the following sequence of drifts:
\begin{equation}\label{eq:ex2drift1}
F^{(N)}(\vec{m})=\begin{pmatrix}
-\vec{m}_1\mathbb{E}\left[p_1(1-p_s(X_1+X_2))\right]+\vec{m}_2\mathbb{E}\left[p_2 p_s(X_1+X_2)\right] \\
-\vec{m}_2\mathbb{E}\left[p_2 p_s(X_1+X_2)\right]+\vec{m}_1\mathbb{E}\left[p1(1-p_s(X_1+X_2))\right] \\
\end{pmatrix}
\end{equation}
in which $X_1 \sim B(N\vec{m}_1 ,p_1)$ and $X_2 \sim B(N\vec{m}_2 ,p_2)$. This simplifies to:
\[F^{(N)}(\vec{m})=\begin{pmatrix}
-\vec{m}_1 p_1(1-(1-\frac{p_1}{2})^{N\vec{m}_1}(1-\frac{p_2}{2})^{N\vec{m}_2})+\vec{m}_2 p_2 (1-\frac{p_1}{2})^{N\vec{m}_1}(1-\frac{p_2}{2})^{N\vec{m}_2} \\
\vec{m}_1 p_1(1-(1-\frac{p_1}{2})^{N\vec{m}_1}(1-\frac{p_2}{2})^{N\vec{m}_2})-\vec{m}_2 p_2 (1-\frac{p_1}{2})^{N\vec{m}_1}(1-\frac{p_2}{2})^{N\vec{m}_2} \\
\end{pmatrix}
\]

It can be shown that the inequality below is always satisfied for $\vec{m},\vec{m}'\in\Delta$:
\[\left\lvert F^{(N)}(\vec{m}')-F^{(N)}(\vec{m})  \right\rvert\leq \sqrt{2}\,\left\lvert\vec{m}'-\vec{m}\right\rvert\,,\]
which proves that $F^{(N)}$ are Lipschitz continuous on $\Delta$. In the same manner, it can be shown that for all $\vec{m}\in\Delta$, $\lvert F^{(N)}(\vec{m})\rvert \leq \sqrt{2}$.

Therefore it is safe to assume that $F^{(N)}$ satisfies both {\bf smoothness} and {\bf boundedness}. Moreover, for any $p_1,p_2>0$ and $\vec{m}\in\Delta$ we have:
\begin{equation}\label{eq:ex2result1}
F^{*}(\vec{m})=\lim_{N\rightarrow\infty}F^{(N)}(\vec{m})=\begin{pmatrix}
-p_1 \vec{m}_1\\
p_1 \vec{m}_1
\end{pmatrix}
\end{equation}
which shows that {\bf limit existence} is also satisfied.

\end{example}

\subsection{Approximations of mean field interaction models}
In Section~\ref{subsec:popprocess} we saw how a population process which satisfies the clock independence assumption can be derived. In Section~\ref{subsec:drift} we derived the drift from the population process, and introduced some conditions ({\bf smoothness}, {\bf boundedness} and {\bf limit existence}) which can hold for the drift. In this section, we explain how the drift satisfying all these conditions can be used to derive a deterministic approximation for the behaviour of the population process.

\begin{definition}[Deterministic Approximations]\label{def:approx}
For $N\geq 1$, let $F^{(N)}$ be a drift for which {\bf smoothness} holds. Let $\{g:\mathbb{R}_{\geq 0}\rightarrow\Delta\}$ be the class of functions from $\mathbb{R}_{\geq 0}$ (continuous time) to $\Delta$. Then $\Phi\subset\{g:\mathbb{R}_{\geq 0}\rightarrow\Delta\}$ is the set of deterministic approximations for which every $\phi\in\Phi$ at time $t\in\mathbb{R}_{\geq 0}$ satisfies the following system of ordinary differential equations (ODEs):
\begin{equation*}
\frac{d}{dt}\phi(t)=F^{(N)}(\phi(t)).
\end{equation*}
\end{definition}

Consider $\phi(0)=\phi_0$ to be the initial condition. According to {\bf smoothness} and based on the Picard-Lindel\"{o}f theorem (Theorem \ref{the:piclin}) since $F^{(N)}$ is Lipschitz continuous, there exists a unique solution to the above system of ODEs. Therefore in the set $\Phi$ there is a unique element $\phi^{(N)}(t)$ which satisfies $\phi^{(N)}(0)=\phi_0$, called the {\it deterministic approximation} for which the following theorem holds.

\begin{theorem}[Mean Field Approximation, cf.~\cite{benaim2008class}, Theorem 1]\label{theo1} For $N\geq 1$, let $\left\{\bar{M}^{(N)}(t)\right\}$ be a sequence of normalised population processes. Let $\left\{F^{(N)}\right\}$ be the corresponding drifts which satisfy {\bf smoothness} and {\bf boundedness}. Let $\left\{\phi^{(N)}(t)\right\}$ be the corresponding sequence of deterministic approximations. Then there exist real constants $c_1,c_2>0$ for which for any time horizon $T\in\mathbb{R}_{\geq 0}$:
\[\sup_{0\leq t\leq T}\left\lvert \bar{M}^{(N)}(t)-\phi^{(N)}(t)\right\rvert\leq \exp(c_1 T)\left(\left\lvert \bar{M}^{(N)}(0)-\phi_0\right\rvert+\mathcal{M}^{(N)}(T)\right)\]
in which $\mathcal{M}^{(N)}(T)$ is a stochastic process satisfying the following conditions:
\[\mathbb{E}\left[\mathcal{M}^{(N)}(T)\right]=0\]
and:
\[\mathbb{E}\left[\mathcal{M}^{(N)}(T)^2\right]\leq\frac{c_2 T}{N}.\]
\end{theorem}

Assuming that {\bf limit existence} holds, define the {\it limit system of ODEs} as: $\phi'(t)=F^{*}(\phi(t))$ with initial condition $\phi_0$. The following result is a direct consequence of Theorem~\ref{theo1}.

\begin{corollary}\label{cor1} 
Consider the assumptions of Theorem~\ref{theo1}, and assume {\bf limit existence} holds. Let $\phi^{*}(t)$ be the solution to the corresponding limit system of ODEs. Assume $\lim_{N\rightarrow\infty} \left\lvert \bar{M}^{(N)}(0)-\phi_0\right\rvert^2=0$. Then for any finite time horizon $T<\infty$:
\[\lim_{N\rightarrow \infty}\mathbb{E}\left[\sup_{0\leq t\leq T}\left\lvert \bar{M}^{(N)}(t)-\phi^{*}(t)\right\rvert^2\right]=0.\]
\end{corollary}

\begin{remark}
Based on Theorem~\ref{theo1} proving convergence in mean is straightforward, meaning that for all $T$:
\[\lim_{N\rightarrow\infty}\mathbb{E}\left[\sup_{0\leq t\leq T}\left\lvert \bar{M}^{(N)}(t)-\phi^{*}(t)\right\rvert\right]=0.\]
\end{remark}

In the following example, we derive the limit system of ODEs for the system discussed in Example~\ref{ex:ex2} and describe its solutions.
\begin{example}\label{ex:ex3}
Consider the drift given in (\ref{eq:ex2result1}). The limit system of ODEs for the system described in Example \ref{ex:interf1} is:
\[
\frac{d}{dt}\phi^{*}_1(t)= -p_1\phi^{*}_1(t)~,~
\frac{d}{dt}\phi^{*}_2(t)= p_1\phi^{*}_1(t).
\]
which together with the initial condition $\phi^{*}(0)$ has the following general solution:
\[
\phi^{*}_1(t)= \phi^{*}_1(0)\, e^{-p_1 t}~,~
\phi^{*}_2(t)= -\phi^{*}_1(0)\, e^{-p_1 t}+1.
\]
Obviously the solution heavily depends on the initial values $\phi^{*}(0)$, but the system has a global attractor at $(0,1)$.
\end{example}

\section{Propagation of chaos and the mean drift}\label{sec:propagation}
Corrollary~\ref{cor1} justifies the use of drift for finding the approximation in cases where the number of agents N is unboundedly large. However, for bounded $N$ the upper bound on the error found by Theorem~\ref{theo1} rarely satisfies one's expectations. In this section we explore the possibility of using the alternative ODEs in~(\ref{eq:ode2}). We explain the notion of propagation of chaos, and show how it relates to what we call the {\it mean drift}.

For a set $E$ let $M(E)$ denote the set of probability measures on $E$. Let the set $\mathcal{S}$ be defined as in Section~\ref{subsec:independentevent}, and for $s,s'\in\mathcal{S}$ define the distance between $s$ and $s'$ as 
\[d(s,s')=
\begin{cases}
0,~\text{if}~s=s'\\
2,~\text{if}~s\neq s',\\
\end{cases}\]
which makes the pair $(\mathcal{S},d)$ a metric space, with the implication that $\mathcal{S}^N$ is also metrizable which allows the definition that follows.

\begin{definition}[$\rho$-chaotic Sequence]\label{def:chaotic}
Let $\rho\in M(\mathcal{S})$ be a probability measure. For $N\geq 1$, the sequence $\big\{\rho_N\big\}$ of measures, each in $M(\mathcal{S}^{N})$, is $\rho$-chaotic iff for any natural number $k$ and bounded functions $f_1,\ldots,f_k:\mathcal{S}\rightarrow\mathbb{R}$,
\[\lim_{N\rightarrow\infty}\int_{\mathcal{S}^N}f_1(\vec{v}_1)f_2(\vec{v}_2)\ldots f_k(\vec{v}_k)\rho_{N}(d\vec{v})=\prod_{i=1}^{k}\int_{\mathcal{S}}f_i(s)\,\rho(ds).\]
\end{definition}

In short, a chaotic sequence maintains a form of independence in the observations of separate agents in the limit. In the literature this independence is often called the \textit{propagation of chaos}, and in the context of Bianchi's analysis the \textit{decoupling assumption}.

In the rest of this discussion, consider the finite instant in time $t\in \mathbb{R}_{\geq 0}$ and its close counterpart $\tau\in T_G$ with $\tau=\epsilon\lfloor Dt\rfloor$. Recall the non-normalized mean field interaction model at time $\tau$, $Y^{(N)}(\tau)=\left(\hat{X}_1(\tau),\ldots,\hat{X}_N(\tau)\right)$, which is a random element in $\mathcal{S}^N$. For $N\geq 1$, let $\rho_N\in M(\mathcal{S}^N)$ be the laws (probability distributions) of $Y^{(N)}(\tau)$. We now state a theorem proven by Sznitman~\cite{gottlieb2000markov,sznitman1991topics}. 
For the random element $\vec{v}$ in $\mathcal{S}^N$ define the empirical measure $\varepsilon_N$, a random element in $M(\mathcal{S})$, as follows:
\begin{equation}\label{eq:empire}
\varepsilon_N(\vec{v})=\frac{1}{N}\sum_{i=1}^{N}\delta_{\vec{v}_i}
\end{equation}
where $\delta_{\vec{v}_i}$ is the Dirac measure centred on point $\vec{v}_i$. In what follows we always assume $\vec{v}=Y^{(N)}(\tau)$, and hence write $\varepsilon_N$ instead of $\varepsilon_N\left(Y^{(N)}(\tau)\right)$.
The following is known as Sznitman's result in the literature.
\begin{theorem}[See \cite{sznitman1991topics}, Proposition 2.2]\label{theo:sznitman}
The sequence $\big\{\rho_N\big\}$ of measures is $\rho$-chaotic if and only if the sequence of empirical measures $\{\varepsilon_N\}$ converges to $\delta_\rho$, that is:
\[\rho_N\circ\varepsilon_N^{-1}\rightarrow\delta_\rho.\]
\end{theorem}

Based on the above theorem, it has been shown that the result of Corollary~\ref{cor1} implies that propagation of chaos also occurs in the sequence of distributions of mean field interaction models $Y^{(N)}(\tau)$ (see also corollary 2 of~\cite{benaim2008class}, and a review of similar results in~\cite{duffy2010mean}). 
\begin{corollary}\label{cor2}
Let $\phi^{*}(t)$ satisfy Corollary~\ref{cor1}. Let $\mu\in M(\mathcal{S})$ be a measure which for all points $i\in\mathcal{S}$ satisfies $\mu(i)=\phi^{*}_i(t)$, then the sequence $\big\{\rho_N\big\}$ of distributions of $Y^{(N)}(\tau)$ is $\mu$-chaotic.
\end{corollary}

Let the measure $\mu$ be defined as in Corollary~\ref{cor2}. For $k\in\{0,\ldots,N\}$, we are interested in finding the following probability given that the distributions of $\left\{Y^{(N)}(\tau)\right\}$ form a $\mu$-chaotic sequence:
\[\mathbb{P}\left\{N \bar{M}^{(N)}_i(t)=k\right\}.\]
Using the fact that $Y^{N}(\tau)$ is a mean field interaction model (or is symmetric), we have:
\[\mathbb{P}\left\{N \bar{M}^{(N)}_i(t)=k\right\}=\binom{N}{k}\int_{\mathcal{S}^N}\mathds{1}(\vec{v}_1=i)\ldots \mathds{1}(\vec{v}_k=i)\,\mathds{1}(\vec{v}_{k+1}\neq i)\ldots \mathds{1}(\vec{v}_N\neq i)\,\rho_{N}(d\vec{v})\]
and since $\{\rho_N\}$ is a $\mu$-chaotic sequence, based on Definition~\ref{def:chaotic}:
\[\lim_{N\rightarrow\infty}\mathbb{P}\left\{N \bar{M}^{(N)}_i(t)=k\right\}=\lim_{N\rightarrow\infty}\binom{N}{k}\left(\phi^{*}_i(t)\right)^k\left(1-\phi^{*}_i(t)\right)^{N-k}.\]
This justifies the usage of the following approximation for sufficiently large $N$:
\[\mathbb{P}\left\{N \bar{M}^{(N)}_i(t)=k\right\} \approx \binom{N}{k}\left(\phi^{(N)}_i(t)\right)^k\left(1-\phi^{(N)}_i(t)\right)^{N-k},\]
which following the Poisson approximation theorem can be in turn approximated by:
\begin{equation}\label{eq:chaospoisson}
\mathbb{P}\left\{N \bar{M}^{(N)}_i(t)=k\right\}\approx e^{-N \phi^{(N)}_i(t)}\frac{\left(N\phi^{(N)}_i(t)\right)^k}{k!}.
\end{equation}
In the context of Bianchi's analysis, see~\cite{vvedenskaya2007multiuser} for an implicit application of a similar approximation.

Let $f_\textit{Poisson}(k;\lambda)$ denote the probability density function of a Poisson random variable with rate $\lambda$. Let $f:\Delta\rightarrow\mathbb{R}$ be a function which acts on the random variable $\bar{M}^{(N)}(t)$. A major convenience in using the above terms is their mutual independence in the limit, which at time $t\in \mathbb{R}_{\geq 0}$ allows the approximation of the expected value of $f(\bar{M}^{(N)}(t))$ as:
\begin{equation}\label{eq:poisaverage}
\mathbb{E}\left[f(\bar{M}^{(N)}(t))\right]\approx\sum_{k_1=0}^{\infty}\ldots\sum_{k_I=0}^{\infty} f\left(\frac{k_1}{N},\ldots,\frac{k_I}{N}\right) f_\textit{Poisson}\left(k_1;N\phi^{(N)}_1(t)\right) \ldots f_\textit{Poisson}\left(k_I;N\phi^{(N)}_I(t)\right).
\end{equation}

\subsection{The mean drift}\label{sec:meandrift}
In this part, we explain the approximation of the ODEs in formula~(\ref{eq:ode2}), using result (\ref{eq:poisaverage}). For the occupancy measure $\vec{m}\in \Delta$, the {\it Poisson mean} of the intensity $F_{s,s'}^{(N)}$ is the function $\tilde{F}_{s,s'}^{(N)}:\Delta\rightarrow \mathbb{R}$, where:
\begin{equation}\label{eq:poaverage}
\tilde{F}_{s,s'}^{(N)}(\vec{m})=\sum_{k_1=0}^{\infty}\ldots\sum_{k_I=0}^{\infty} F_{s,s'}^{(N)}\left(\frac{k_1}{N},\ldots,\frac{k_I}{N} \right) f_\textit{Poisson}(k_1;N\vec{m}_{1})\ldots f_\textit{Poisson}(k_I;N\vec{m}_{I})\,.
\end{equation}

Subsequently, the {\it mean drift} $\tilde{F}^{(N)}:\Delta\rightarrow \mathbb{R}^{I}$ is defined as:
\[\tilde{F}^{(N)}(\vec{m})=\sum_{s,s'\in\mathcal{S}}\tilde{F}_{s,s'}^{(N)}(\vec{m})(\vec{e}_{s'}-\vec{e}_s).\]

The Poisson mean of intensities and the mean drift have the following properties:
\begin{itemize}
\item If $F_{s,s'}^{(N)}(\vec{m})=\vec{m}_s\alpha$, where $\alpha$ is a constant or a term $\vec{m}_j$ for some $j\in\mathcal{S}$, then $\tilde{F}_{s,s'}^{(N)}(\vec{m})=F_{s,s'}^{(N)}(\vec{m})$.
\item $\tilde{F}^{(N)}(\vec{m})$ is defined for all $\vec{m}\in\Delta$.
\item Given {\bf smoothness} and {\bf boundedness} of the drift, $\tilde{F}^{(N)}$ is both Lipschitz continuous and bounded on $\Delta$.
\end{itemize}

Based on the derivation of probabilities (\ref{eq:chaospoisson}), it is easy to see that at time $t<\infty$ we have:
\[\mathbb{E}\left[F^{(N)}(\bar{M}^{N}(t))\right]\approx\tilde{F}^{(N)}\left(\phi^{(N)}(t)\right).\]
Moreover, based on (\ref{eq:chaospoisson}), $\mathbb{E}\left[\bar{M}^{N}(t)\right]\approx\phi^{(N)}(t)$, with which the following system of differential equations can be derived from (\ref{eq:ode2}):
\begin{equation}\label{eq:betterode}
\frac{d}{d t}\phi^{(N)}(t)=\tilde{F}^{(N)}\left(\phi^{(N)}(t)\right)
\end{equation}
with the initial condition $\phi^{(N)}(0)=\bar{M}^{(N)}(0)$.

Our construction which relies on the notion of propagation of chaos, means that the differential equations (\ref{eq:betterode}) give better approximations as the system size $N$ grows. This fact is demonstrated by the following theorem.
\begin{theorem}\label{theo:drifts}
For $N\geq 1$, let $\left\{F^{(N)}\right\}$ be the sequence of drifts, and $\left\{\tilde{F}^{(N)}\right\}$ be the corresponding sequence of mean drifts. Assume that the drifts satisfy {\bf smoothness} and {\bf boundedness}. Then for all $\vec{m}\in \Delta$,
\[\lim_{N\rightarrow\infty}\tilde{F}^{(N)}(\vec{m})=F^{*}(\vec{m})\]
almost surely.
\end{theorem}

\begin{example}\label{ex:ex4}
For the system described in Example~\ref{ex:interf1} we find the mean drift, using the drift $F^{(N)}$ described in (\ref{eq:ex2drift1}). The mean drift takes the relatively simple shape which follows:
\begin{equation*}\label{eq:ex2result}
\renewcommand*{\arraystretch}{1.5}
\tilde{F}^{(N)}(\vec{m})=\begin{pmatrix}
\frac{p_1\vec{m}_1+p_2\vec{m}_2}{2} exp\left\{-\frac{p_1 N\vec{m}_1+p_2 N\vec{m}_2}{2}\right\} -p_1 \vec{m}_1 \\
-\frac{p_1\vec{m}_1+p_2\vec{m}_2}{2} exp\left\{-\frac{p_1 N\vec{m}_1+p_2 N\vec{m}_2}{2}\right\} + p_1 \vec{m}_1\\
\end{pmatrix}
\end{equation*}
This can be used to construct the following system of ODEs:
\begin{flalign*}
\frac{d}{dt}\phi^{(N)}_1&= \frac{p_1\phi^{(N)}_1(t)+p_2\phi^{(N)}_2(t)}{2} exp\left\{-\frac{p_1 N\phi^{(N)}_1(t)+p_2 N\phi^{(N)}_2(t)}{2}\right\} -p_1 \phi^{(N)}_1(t)& \\
\frac{d}{dt}\phi^{(N)}_2&= - \frac{p_1\phi^{(N)}_1(t)+p_2\phi^{(N)}_2(t)}{2} exp\left\{-\frac{p_1 N\phi^{(N)}_1(t)+p_2 N\phi^{(N)}_2(t)}{2}\right\}+p_1 \phi^{(N)}_1(t) & 
\end{flalign*}

\begin{figure}[h]
\centering
\scalebox{0.79}{
\pgfplotstableread{test.data}{\test}
\pgfplotstableread{approximation.data}{\approximation}
\pgfplotstableread{simulation.data}{\simulation}
\begin{tikzpicture}
\begin{axis}[
height=8cm,
width=15cm,
ylabel=$\phi_2^{(N)}(1000)$,
xlabel=$N$,
xlabel near ticks,
xticklabels={1,2,5,10,20,50,100,200,400,800,1600},
xtick={1,...,11},
minor y tick num=1,
grid=both,
x tick label style={rotate=90,anchor=east},
legend pos=north west]
\addplot [x={nodes}, y={p},mark=triangle,solid,every mark/.append style={solid, fill=black}] table {\test};
\addplot [x={nodes}, y={p},mark=otimes,loosely dotted,every mark/.append style={solid, fill=gray}] table {\approximation};
\addplot [x={nodes}, y={p},mark=+,dashed] table {\simulation};
\addlegendentry{Solutions of ODEs with drift}
\addlegendentry{Solutions of ODEs with mean drift}
\addlegendentry{Results of the transient analysis}
\end{axis}
\end{tikzpicture}
}
\caption{Comparison between the proportion of nodes in the back-off state at time $t=1000$ ($\phi_2^{(N)}(1000)$) for different network sizes $N$, based on a transient analysis of the Markov models (the solutions of the Chapman-Kolmogorov equations) and a mean field analysis by the ODEs incorporating the mean-drift.}\label{fig:simodecomparison}
\end{figure}
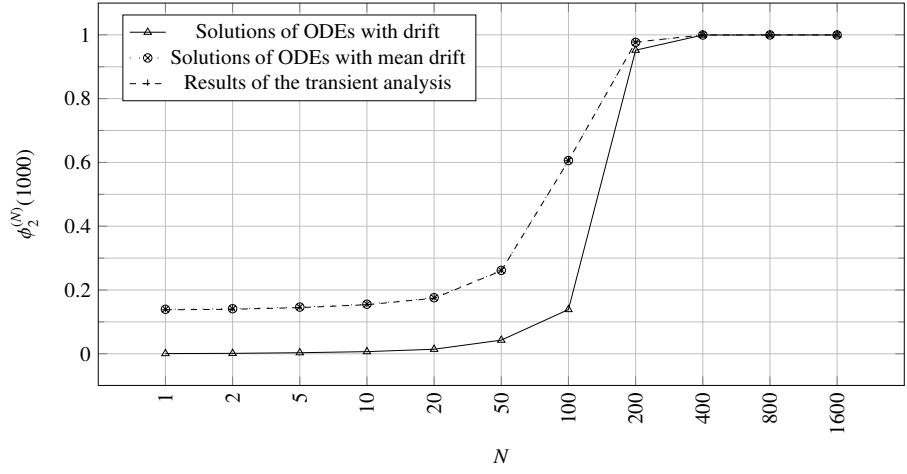
Let $p_1=0.008$ and $p_2=0.05$ and let the initial condition be $\phi_0=(1,0)$, i.e., all the nodes are initially in state $1$. In Fig.~\ref{fig:simodecomparison} the results of solving the ODEs for different values of $N$ are given, and are compared with results from the explicit transient analysis of Markov models with simultaneous updates. Observe that regardless of the size of the system, the approximations derived by the two methods closely match one another.
\end{example}

\section{Proofs}\label{sec:proofs}
Detailed proofs for the theorems and corollaries follow.

\subsection{Proof of Proposition~\ref{pro:driftgen}}\label{subsec:gen}
\begin{lemma}
Let $N\geq 1$ be a natural number, and $\{\bar{M}^{(N)}(t):t\in T_G\}$ be a normalized population process satisfying the clock independence assumption. Let $\{\mathcal{T}_t:t\in T\}$ be a sequence of linear operators, which for functions $f:\Delta^{(N)}\rightarrow\Delta^{(N)}$ and $t\in T$ satisfy:
\[\mathcal{T}_t\,f(\vec{m})=\mathbb{E}\left[f(\bar{M}^{(N)}(t))\mid \bar{M}^{(N)}(0)=\vec{m}\right].\]
then given the infinitesimal generator $A$ of $\bar{M}^{(N)}(t)$, for the identity function $f$ we have $Af=F^{(N)}$.
\end{lemma}
\begin{proof}
Given $\epsilon=\frac{1}{D}$, and using (\ref{eq:generator}):
\begin{flalign*}
Af&=\lim_{t\rightarrow 0}\frac{\mathbb{E}\left[\bar{M}^{(N)}(t)\mid \bar{M}^{(N)}(0)\right]-\bar{M}^{(N)}(0)}{t}\\
&=\lim_{t\rightarrow 0}\frac{\mathbb{E}\left[\bar{M}^{(N)}(t)-M^{(N)}(\epsilon\lfloor Dt\rfloor)\mid M^{(N)}(0)\right]+\mathbb{E}\left[M^{(N)}(\epsilon\lfloor Dt\rfloor)\mid M^{(N)}(0)\right]-M^{(N)}(0)}{t}\\
&=\lim_{t\rightarrow 0}\frac{0+\sum_{s=1}^{\lfloor Dt\rfloor}\mathbb{E}\left[M^{(N)}(\epsilon\, s)-M^{(N)}(\epsilon\,(s-1))\mid M^{(N)}(\epsilon\,(s-1))\right]}{t}\\
&=\lim_{t\rightarrow 0}\frac{\sum_{s=1}^{\lfloor Dt\rfloor}\hat{F}^{(N)}(M^{(N)}(\epsilon\,(s-1)))}{t}\\
&=\lim_{t\rightarrow 0}\frac{\sum_{s=1}^{\lfloor Dt\rfloor}\epsilon\, F^{(N)}(M^{(N)}(\epsilon\,(s-1)))}{t}
\end{flalign*}
Based on the definition of $\bar{M}^{(N)}(t)$ the numerator can be written as:
\begin{flalign*}
\sum_{s=1}^{\lfloor Dt\rfloor}\epsilon\, F^{(N)}(M^{(N)}(\epsilon\,(s-1)))&=\int^{Dt}_{0}F^{(N)}(\bar{M}^{(N)}(\epsilon\, s))\epsilon\, ds\\
&\quad\quad-\int^{Dt}_{\lfloor Dt\rfloor}F^{(N)}(\bar{M}^{(N)}(\epsilon\, s))\epsilon\, ds
\end{flalign*}
Since the construction of $F^{(N)}$ involves the clock independence assumption, $t\geq \epsilon$, in which case:
\begin{flalign*}
Af&=\lim_{t\rightarrow 0}\frac{\int^{Dt}_{0}F^{(N)}\left(\bar{M}^{(N)}(\epsilon\, s)\right)\epsilon\, ds}{t}\\
&=\lim_{t\rightarrow 0}\frac{\int^{t}_{0}F^{(N)}\left(\bar{M}^{(N)}(r)\right)dr}{t}\\
&=F^{(N)}\left(\bar{M}^{(N)}(0)\right)=F^{(N)}(\vec{m}).
\end{flalign*}
\end{proof}

Thus, we have proven that $\left(f,F^{(N)}\right)\in A$, and subsequently according to (\ref{eq:generator2}) the following equation holds:
\[
\mathbb{E}\left[\bar{M}^{(N)}(t)\mid \bar{M}^{(N)}(0)\right]-\bar{M}^{(N)}(0)=\int_{0}^{t}\mathbb{E}\left[F^{(N)}\left(\bar{M}^{(N)}(s)\right)\mid \bar{M}^{N}(0)\right]ds.
\]

\subsection{Proof of Theorem~\ref{theo1}}
The following proof and the proof for Corollary\ref{cor1} are largely based on a similar proof in~\cite{ethier2009markov} which has been repeated in many other work. Here, our main goal is to give a simple enough version of the proof, yet without big unexplained leaps.

Given the random processes $W^{(N)}_{s,s'}(t)$ defined in Section~\ref{subsec:drift} and satisfying equation~(\ref{eq:counter}), define $Z^{(N)}_{s,s'}(t)$ as the total random number of agents taking a transition from state $s\in\mathcal{S}$ to $s'\in\mathcal{S}$ in the time interval $[0,t]$:
\[Z^{(N)}_{s,s'}(t)=\lim_{D\rightarrow\infty}\sum_{i=1}^{\lfloor Dt\rfloor}W^{(N)}_{s,s'}(i\, \epsilon),\]
then the definition of the intensity of transitions from $s$ to $s'$ (also ultimately based on $W^{(N)}_{s,s'}(t)$) implies the following:
\begin{equation}\label{eq:counts}
\mathbb{E}\left[Z^{(N)}_{s,s'}(t)\right]=N\int_{0}^{t}F^{(N)}_{s,s'}\left(\bar{M}^{(N)}(s)\right)ds.
\end{equation}
The process $Z^{(N)}_{s,s'}(t)$ is Markov in $T_G$, and its sample paths are increasing functions which are right-continuous with left limits.

We now present a result given in~\cite{kurtz1981approximation} (theorem 7.1). For each process $Z^{(N)}_{s,s'}(t)$ associated with an intensity from $s$ to $s'$, assume an increasing Markov jump process (a counting process) $\mathcal{N}(t)$ for which $\mathbb{E}[\mathcal{N}(t)]=ct$ for some $c\in\mathbb{R}_{\geq 0}$. Then Following~\cite{watanabe1964discontinuous}, the process $\mathcal{N}(t)$ is a Poisson process. Moreover, following~\cite{volkonskii1958random,kurtz1981approximation}, since the processes $Z^{(N)}_{s,s'}(t)$ and $\mathcal{N}(t)$ are Markov (in $T_G$ and in $\mathbb{R}$ respectively), and increasing, for $t\in T_G$ there exist stopping times $\tau(t)$ which satisfy:
\begin{equation}\label{eq:stptime}
Z^{(N)}_{s,s'}(t)=\mathcal{N}(\tau(t)).
\end{equation}
Essentially, these stopping times maintain the equality of the values over the discrete jumps of processes. Following~\cite{meyer1971demonstration,
brown1988simple} for all pairs $s,s'\in\mathcal{S}$ the resulting Poisson processes are mutually independent. This allows freedom in combining these processes to describe the behaviour of the process $\bar{M}^{(N)}(t)$.

From (\ref{eq:counts}) and (\ref{eq:stptime}) we may write:
\[\mathbb{E}\left[\mathcal{N}(\tau(t))\right]=N\int_{0}^{t}F^{(N)}_{s,s'}\left(\bar{M}^{(N)}(s)\right)ds.\]
Let $\{\mathcal{Y}(t):t\in\mathbb{R}\}$ be a unit Poisson process (see section~\ref{subsec:poisson}). Using equation (\ref{eq:unitpoisson}) we may exchange time and intensity to write $\mathcal{N}(\tau(t))$ in terms of unit Poisson processes as $\mathcal{N}(\tau(t))=\mathcal{Y}(\mathbb{E}[\mathcal{N}(\tau(t))])$, so we may rewrite  equation (\ref{eq:counts}) as:
\begin{equation}\label{eq:unitrep1}
Z^{(N)}_{s,s'}(t)=\mathcal{Y}\left(N\int_{0}^{t}F^{(N)}_{s,s'}\left(\bar{M}^{(N)}(s)\right)ds\right)
\end{equation}

Note that for the process $\bar{M}^{(N)}(t)$ we have:

\[\bar{M}^{(N)}(t)=\bar{M}^{(N)}(0)+\sum_{s,s'\in\mathcal{S}\times\mathcal{S}}\frac{1}{N}Z^{(N)}_{s,s'}(t)\,(e_{s'}-e_{s}),\]

which combined with equation (\ref{eq:unitrep1}) results in the following well-known format, often called the Poisson representation of the population process:
\begin{equation}\label{eq:unitrep2}
\bar{M}^{(N)}(t)=\bar{M}^{(N)}(0)+\sum_{k=(s,s')\in\mathcal{S}\times\mathcal{S}}\frac{1}{N}\mathcal{Y}_k\left(N\int_{0}^{t}F_k^{(N)}(\bar{M}^{(N)}(s))ds\right)\,(e_{s'}-e_{s})
\end{equation}

For all pairs $k=(s,s')\in\mathcal{S}\times\mathcal{S}$ consider the following process: 
\[\tilde{\mathcal{Y}}_k\left(N\int_{0}^{t}F_k^{(N)}(\bar{M}^{(N)}(s))ds\right)=\mathcal{Y}_k\left(N\int_{0}^{t}F_k^{(N)}(\bar{M}^{(N)}(s))ds\right)-N\int_{0}^{t}F_k^{(N)}(\bar{M}^{(N)}(s))ds\]
which is a compensated unit Poisson process and hence a martingale (see Section~\ref{subsec:poisson}). From here:
\begin{flalign*}
\bar{M}^{(N)}(t)&=\bar{M}^{(N)}(0)+\sum_{k=(s,s')\in\mathcal{S}\times\mathcal{S}}\frac{1}{N}\left[\tilde{\mathcal{Y}}_k\left(N\int_{0}^{t}F_k^{(N)}(\bar{M}^{(N)}(s))ds\right)+N\int_{0}^{t}F_k^{(N)}(\bar{M}^{(N)}(s))ds\right](e_{s'}-e_{s})\\
&=\bar{M}^{(N)}(0)+\sum_{k=(s,s')\in\mathcal{S}\times\mathcal{S}}\frac{1}{N}\tilde{\mathcal{Y}}_k\left(N\int_{0}^{t}F_k^{(N)}(\bar{M}^{(N)}(s))ds\right)\,(e_{s'}-e_{s})+\int_{0}^{t}F^{(N)}(\bar{M}^{(N)}(s))ds.
\end{flalign*}
In this decomposition, the process:
\[\mu^{(N)}(t)=\sum_{k=(s,s')\in\mathcal{S}\times\mathcal{S}}\frac{1}{N}\tilde{\mathcal{Y}}_k\left(N\int_{0}^{t}F_k^{(N)}(\bar{M}^{(N)}(s))ds\right)\,(e_{s'}-e_{s})\]
is a martingale and using (\ref{eq:doob}) (Doob's inequality) for $\alpha=2$ we have:
\[\mathbb{E}\left[\sup_{t\leq T}\mu^{(N)}(t)^2\right]\leq 4\mathbb{E}\left[\mu^{(N)}(T)^2\right],\]
and based on {\bf boundedness}, $F_k^{(N)}$ are bounded, say by constant $A$, therefore:
\[\mathbb{E}\left[\mu^{(N)}(T)^2\right]\leq\frac{1}{N}\sum_{i\in\mathcal{S}}e_{i} A T.\]
Set $c_2=4A\lvert I\rvert$. Next, we write the system of ODEs in its integral form:
\[\phi^{(N)}(t)=\phi(0)+\int_{0}^{t}F^{(N)}\left(\phi^{(N)}(s)\right)ds.\]
For $T\geq 0$:
\begin{flalign*}
\sup_{t\leq T}\left\lvert \bar{M}^{(N)}(t)-\phi^{(N)}(t)\right\rvert &\leq  \left\lvert \bar{M}^{(N)}(0)-\phi^{(N)}(0)\right\rvert+\sup_{t\leq T}\left\lvert\,\mu^{(N)}(t)\right\rvert\\
&\quad+\sup_{t\leq T}\left\lvert \int_{0}^{t}F^{(N)}(\bar{M}^{(N)}(s))ds-\int_{0}^{t}F^{(N)}(\phi^{(N)}(s))ds\right\rvert\\
&\leq \left\lvert \bar{M}^{(N)}(0)-\phi^{(N)}(0)\right\rvert+\sup_{t\leq T}\left\lvert\,\mu^{(N)}(t)\right\rvert+\sup_{t\leq T}\int_{0}^{t}\left\lvert F^{(N)}(\bar{M}^{(N)}(s))-F^{(N)}(\phi^{(N)}(s))\right\rvert ds
\end{flalign*}

Since according to assumptions each intensity $F_k^{(N)}$ is Lipschitz continuous, say by constant $L_k$, there exists a constant $L=max(L_k:k\in\mathcal{S}\times \mathcal{S})$, for which:
\begin{flalign*}
\sup_{t\leq T}\left\lvert \bar{M}^{(N)}(t)-\phi^{(N)}(t)\right\rvert &\leq  \left\lvert \bar{M}^{(N)}(0)-\phi^{(N)}(0)\right\rvert+\sup_{t\leq T}\left\lvert\,\mu^{(N)}(t)\right\rvert+L\sup_{t\leq T}\int_{0}^{t}\left\lvert \bar{M}^{(N)}(s)-\phi^{(N)}(s)\right\rvert ds\\
&\leq \left\lvert \bar{M}^{(N)}(0)-\phi^{(N)}(0)\right\rvert+\sup_{t\leq T}\left\lvert\,\mu^{(N)}(t)\right\rvert+L\int_{0}^{T}\sup_{t\leq s}\left\lvert \bar{M}^{(N)}(t)-\phi^{(N)}(t)\right\rvert ds
\end{flalign*}
Let $f(T)=\sup_{t\leq T}\lvert \bar{M}^{(N)}(t)-\phi^{(N)}(t)\rvert$, using Gr\"onwall's inequality (Lemma~\ref{lem:gronwall}):
\[\sup_{t\leq T}\left\lvert \bar{M}^{(N)}(t)-\phi^{(N)}(t)\right\rvert\leq \exp(L T)\left(\left\lvert \bar{M}^{(N)}(0)-\phi^{(N)}(0)\right\rvert+\sup_{t\leq T}\left\lvert\,\mu^{(N)}(t)\right\rvert\right).\]
Setting $c_1=L$, $\mathcal{M}^{(N)}(T)=\sup_{t\leq T}\left\lvert\,\mu^{(N)}(t)\right\rvert$ and $c_2$ as found earlier then proves the theorem.

\subsection{Proof of Corollary~\ref{cor1}}
In what follows we consider the assumptions of corollary \ref{cor1} to hold. Our presentation takes several properties of topological spaces for granted, with the understanding that detailed proofs are available to the reader and we are allowed to avoid discussions aimed at overcoming such difficulties.

\begin{lemma}\label{lem:conv}
Let the sequence of drifts $\{F^{(N)}\}$ converge uniformly to the bounded function $F^{*}$, that is:
\[\lim_{N\rightarrow\infty}\sup_{\vec{m}\in\Delta^{(N)}}\left\lvert F^{(N)}(\vec{m})-F^{*}(\vec{m})\right\rvert=0.\]
Then the sequence of solutions $\{\phi^{(N)}\}$ of respective ODEs converges to $\phi^{*}(t)$, the solution of the limit system of ODEs:
\[\lim_{N\rightarrow\infty} \phi^{(N)}(t)=\phi^{*}(t).\]
\end{lemma}
\begin{proof}
Let $\phi^{(N)}(t)$ be the solution of the initial value problem $\phi'(t)=F^{(N)}(\phi(t))$ with $\phi(0)=\phi^{(N)}(0)=\phi_0$ at time $t\in\mathbb{R}_{\geq 0}$. It is possible to make the following assertions based on the unofirm convergence assumption:
\[\lim_{N\rightarrow\infty}\sup_{t}\,\left\lvert F^{(N)}(\phi^{(N)}(t))-F^{*}(\phi^{(N)}(t))\right\rvert=0~~\Rightarrow~~\lim_{N\rightarrow\infty}\sup_{t}\,\left\lvert \phi'^{(N)}(t)-F^{*}(\phi^{(N)}(t))\right\rvert=0.\]
The second limit implies that as $N\rightarrow\infty$ at every point $t\in\mathbb{R}_{\geq 0}$, $\phi^{(N)}(t)$ is a solution to the limit system of ODEs. Given that $lim_{N\rightarrow\infty}\phi^{(N)}(0)=\phi^{*}(0)=\phi_0$ and $F^{(N)}$ are Lipschitz, the existence and uniqueness of the solution implies that for all $t$:
\begin{equation*}
\lim_{N\rightarrow\infty}\phi^{(N)}(t)=\phi^{*}(t).
\end{equation*}
\end{proof}

We now return to the proof of Corollary~\ref{cor1}. Using Theorem~\ref{theo1} the following inequality holds:
\begin{equation*}
\sup_{0\leq t\leq T}\left\lvert \bar{M}^{(N)}(t)-\phi^{(N)}(t)\right\rvert\leq \exp(c_1 T)\left(\mathcal{M}^{(N)}(T)+\left\lvert \bar{M}^{(N)}(0)-\phi_0\right\rvert\right)
\end{equation*}
therefore:
\begin{flalign*}
\sup_{0\leq t\leq T}\left\lvert \bar{M}^{(N)}(t)-\phi^{(N)}(t)\right\rvert^2 &\leq \exp(c_1 T)^2\left(\mathcal{M}^{(N)}(T)+\left\lvert \bar{M}^{(N)}(0)-\phi_0\right\rvert\right)^2
\end{flalign*}
and,
\begin{flalign*}
\mathbb{E}\left[\sup_{0\leq t\leq T} \left\lvert\bar{M}^{(N)}(t)-\phi^{(N)}(t)\right\rvert^2\right] &\leq \exp(c_1 T)^2\left(\mathbb{E}[\mathcal{M}^{(N)}(T)^2]+\left\lvert\bar{M}^{(N)}(0)-\phi_0\right\rvert ^2 \right)\\
&\leq \exp(c_1 T)^2\left(c_2 T\frac{1}{N}+\left\lvert \bar{M}^{(N)}(0)-\phi_0\right\rvert^2\right).
\end{flalign*}
Hence since $\bar{M}^{(N)}(0)$ converges to $\phi_0$ in mean-square, for $T<\infty$:
\begin{equation*}
\lim_{N\rightarrow\infty}\mathbb{E}\left[\sup_{0\leq t\leq T}\left\lvert \bar{M}^{(N)}(t)-\phi^{(N)}(t)\right\rvert^2\right]=0,
\end{equation*}
Which based on Lemma~\ref{lem:conv}, implies that:
\begin{equation*}
\lim_{N\rightarrow\infty}\mathbb{E}\left[\sup_{0\leq t\leq T}\left\lvert \bar{M}^{(N)}(t)-\phi^{*}(t)\right\rvert^2\right]=0.
\end{equation*}

\subsection{Proof of Corollary~\ref{cor2}}
If we show that:
\begin{equation}\label{eq:proofcor2}
\rho_N\circ\varepsilon_N^{-1}\rightarrow\delta_\mu,
\end{equation}
then based on Theorem~\ref{theo:sznitman} the sequence must be $\mu$-chaotic. Observe that according to Markov's inequality for any $c>0$ and every $A\subseteq\mathcal{S}$:
\begin{equation}\label{eq:markov}
\mathbb{P}_N\left\{\left\rvert\varepsilon_N(A)-\mu(A)\right\rvert\geq c\right\}\leq \frac{\mathbb{E}_N\left[\lvert\varepsilon_N(A)-\mu(A)\rvert \right]}{c}
\end{equation}
where $\mathbb{E}_N$ and $\mathbb{P}_N$ denote values calculated according to measure $\rho_N$. By equation (\ref{eq:normal}) (definition of $M^{(N)}(t)$) and definition of $\mu$:
\begin{flalign*}
\lim_{N\rightarrow\infty}\mathbb{E}_N\left[\lvert\varepsilon_N(A)-\mu(A)\rvert\right]&=\lim_{N\rightarrow\infty}\mathbb{E}\left[
\left\lvert\frac{1}{N}\sum_{i=1}^{N}\delta_{Y^{(N)}_i(t)}(A)-\mu(A)\right\rvert\right]\\
&=
\lim_{N\rightarrow\infty}\mathbb{E}\left[
\left\lvert\sum_{j=1}^{I}M^{(N)}_j(t) \mathds{1}(j\in A)-\sum_{j=1}^{I}\phi^{*}_j(t) \mathds{1}(j\in A)\right\rvert\right]\\
&\leq
\lim_{N\rightarrow\infty}\mathbb{E}\left[
\sum_{j=1}^{I}\left\lvert M^{(N)}_j(t)-\phi^{*}_j(t)\right\rvert \mathds{1}(j\in A)\right].
\end{flalign*}
According to Theorem~\ref{theo1} and Corollary~\ref{cor1}, for all $t<\infty$, $\left\{\bar{M}^{(N)}(t)\right\}$ converges to $\phi^{*}(t)$ in both mean and mean square, therefore for every $A\subseteq\mathcal{S}$:
\[\lim_{N\rightarrow\infty}\mathbb{E}\left[
\sum_{j=1}^{I}\left\lvert M^{(N)}_j(t)-\phi^{*}_j(t)\right\rvert \mathds{1}(j\in A)\right]=0.\]
Based on (\ref{eq:markov}) for any $c>0$ and every $A\subseteq\mathcal{S}$:
\[\lim_{N\rightarrow\infty}\mathbb{P}_N\left\{\left\lvert\varepsilon_N(A)-\mu(A)\right\rvert\geq c\right\}=0,\]
which implies that for $t<\infty$, $\varepsilon_N$ converges in probability to constant random element $\mu$. Since convergence in probability implies convergence in distribution, we have shown that the measure $\rho_N\circ\varepsilon_N^{-1}$ converges to $\delta_\mu$.

\subsection{Proof of Theorem~\ref{theo:drifts}}
In order to make the proof shorter and simpler to follow, we first make an (optional) assumption, which follows. A proof applying to the general mean drift defined in Section~\ref{sec:meandrift} would use the same ideas.

\vspace{2mm}

{\bf Simplicity}: For each $s,s'\in\mathcal{S}$, there is one $j\in\mathcal{S}$ for which $Q^{(N)}_{s,s'}(\vec{m})=Q^{(N)}_{s,s'}(\vec{m}_j\,\vec{e}_j)$, i.e., the value $Q^{(N)}_{s,s'}$ is only determined by entry $j$ of an occupancy measure $\vec{m}$.

\vspace{2mm}

Suppose that $Q_{s,s'}^{(N)}$ satisfies {\bf simplicity}, fixing $j$. The {\it Poisson mean} of the intensity $F_{s,s'}^{(N)}$ is the function $\tilde{F}_{s,s'}^{(N)}:\Delta\rightarrow \mathbb{R}$, where:
\begin{equation}\label{eq:simplepoaverage}
\tilde{F}_{s,s'}^{(N)}(\vec{m})=
\vec{m}_s\sum_{k=0}^{N}Q_{s,s'}^{(N)}\left(\frac{k}{N} \vec{e}_{j}\right) f_\textit{Poisson}(k;N\vec{m}_{j})
\end{equation}

We first prove that for all $\vec{m}\in\Delta$ for which $F_{s,s'}^{(N)}(\vec{m})$ is defined:
\[\lim_{N\rightarrow\infty}\tilde{F}_{s,s'}^{(N)}(\vec{m})=\lim_{N\rightarrow\infty}F_{s,s'}^{(N)}(\vec{m}),\]
almost surely.

Based on the assumptions of the theorem, {\bf simplicity} holds, fixing a $j$ such that $F^{(N)}_{s,s'}(\vec{m})=F^{(N)}_{s,s'}(\vec{m_j}\,\vec{e}_j)$. For $1\leq k\leq N$ define the sequence of i.i.d. Poisson random variables $X_k$, with rate $\vec{m}_j$. Define:
\[S_N=\sum_{k=1}^{N}X_k.\]
Then $S_N$ is Poisson distributed with rate $N\vec{m}_j$:
\begin{equation}\label{eq:sn}
\mathbb{P}\left\{S_N=k\right\}=\frac{e^{-N\vec{m}_{j}}}{k!}(N\vec{m}_j)^k
\end{equation}
Following the weak law of large numbers (Theorem~\ref{theo:wlln}), for any $\varepsilon>0$ we have:
\begin{equation}\label{eq:prop1}
\lim_{N\rightarrow\infty}\mathbb{P}\bigg\{\left\lvert\frac{S_N}{N}-\vec{m}_j\bigg\rvert\geq\varepsilon\right\}=0
\end{equation}

Take an arbitrarily small value $\varepsilon>0$. For $N\rightarrow\infty$, the set $\left\{0,\frac{1}{N},\ldots,\frac{N(1+\varepsilon)}{N}\right\}$ is dense in the interval $[0,1+\varepsilon]$. Since $\vec{m}_j\in[0,1]$ and the above set is dense, there exists a largest non-empty neighbourhood $\left\{\frac{k_1}{N},\ldots,\frac{k_m}{N}\right\}\subset \left\{0,\frac{1}{N},\ldots,\frac{N(1+\varepsilon)}{N}\right\}$ for which for all $k_1\leq k_i \leq k_m$, $\left\lvert\frac{k_i}{N}-\vec{m}_j\right\rvert<\varepsilon$ holds. Given such a subset, based on (\ref{eq:prop1}):
\begin{equation*}
\lim_{N\rightarrow\infty}\mathbb{P}\left\{\frac{S_N}{N}\in\bigg\{\frac{k_1}{N},\ldots,\frac{k_m}{N}\bigg\}\right\}=1,
\end{equation*}
which then implies:
\begin{equation}\label{eq:almostsurebunch}
\lim_{N\rightarrow\infty}\sum_{k=k_1}^{k_m}\mathbb{P}\{S_N=k\}=1.
\end{equation}

Next consider the limit:
\begin{flalign*}
\lim_{N\rightarrow\infty}\tilde{F}_{s,s'}^{(N)}(\vec{m})&=\lim_{N\rightarrow\infty}\vec{m}_s\sum_{k=0}^{\infty}Q_{s,s'}^{(N)}\left(\frac{k}{N} \vec{e}_{j}\right)\frac{e^{-N\vec{m}_{j}}}{k!}(N\vec{m}_j)^k,
\end{flalign*}
since for all $\vec{m}\in\Delta$, $Q_{s,s'}^{(N)}(\vec{m})\leq 1$ (is finite), based on (\ref{eq:sn}) and (\ref{eq:almostsurebunch}) we have:
\begin{flalign*}
\lim_{N\rightarrow\infty}\tilde{F}_{s,s'}^{(N)}(\vec{m})&=\lim_{N\rightarrow\infty}\vec{m}_s\sum_{k=k_1}^{k_m}Q_{s,s'}^{(N)}\left(\frac{k}{N} \vec{e}_{j}\right)\frac{e^{-N\vec{m}_{j}}}{k!}(N\vec{m}_j)^k,~\text{a.s.}
\end{flalign*}
Moreover, since the interval $2\varepsilon$ is arbitrarily small, and according to assumption $Q_{s,s'}^{(N)}$ is Lipschitz continuous; for all $k_1\leq k_i\leq k_m$ we have $Q_{s,s'}^{(N)}(\frac{k_i}{N}\vec{e}_j)\simeq Q_{s,s'}^{(N)}(\vec{m}_j\, \vec{e}_{j})$. This, together with {\bf limit existence} implies that:
\begin{flalign*}
\lim_{N\rightarrow\infty}\tilde{F}_{s,s'}^{(N)}(\vec{m})&=\lim_{N\rightarrow\infty}\vec{m}_s Q_{s,s'}^{(N)}(\vec{m}_j\, \vec{e}_{j})\\
&=F^{*}_{s,s'}(\vec{m}),~\text{a.s.}
\end{flalign*}

Given {\bf boundedness}, the dominated convergence theorem implies:
\[\lim_{N\rightarrow\infty}\sum_{s,s'\in\mathcal{S},s\neq s'}\tilde{F}_{s,s'}^{(N)}(\vec{m})\,(\vec{e}_{s'}-\vec{e}_s)=F^{*}(\vec{m}),~\text{a.s.}\]
which given the definition of the mean drift proves the statement of the theorem.

\section{Conclusion}
In \cite{talebi2015continuous} the authors apply a version of the superposition principle called the {\it Poisson averaging} of the drift while deriving ODEs for wireless sensor networks. The method is used to cope with the ambiguous meaning of fractions which appear in arguments given to functions originally defined on discrete domains, and essentially interpolates the value of the function by interpreting occupancy measures as Poisson arrivals. In this paper we justify this practice by deriving a similar set of ODEs and showing how they relate to other concepts in the mean field theory of Markov processes. 

The result is the introduction of the mean drift, a concept that supports the analysis of bounded systems given that the mean field approximation applies. We maintain that within our formal framework the approximation theorems hold for systems with an infinite number of agents. We bridge the gap between the ODEs derived by using the mean drift and the ODEs describing the behaviour of the system in the limit by proving Theorem~\ref{theo:drifts}, which states that under a familiar set of conditions (smoothness and boundedness of the drift), the sequence of mean drifts converges to the limit of the sequence of drifts due to the law of large numbers.

We expect that for middle-sized systems, deriving ODEs using the mean drift gives far better approximations for the behaviour of the systems. As such, the current work provides a stepping stone for future efforts to apply the mean field analysis to the design and performance evaluation of distributed systems.

\vspace{4mm}
\textbf{Acknowledgments.} The research from DEWI project (www.dewi-project.eu) leading to these results has received funding from the ARTEMIS Joint Undertaking under grant agreement \textnumero~621353.

\bibliography{all}

\end{document}